  \documentclass[12pt,dvips]{amsart}
  \usepackage{amsfonts, amssymb, latexsym, epsfig}
  \usepackage{euler, amsfonts, amssymb, latexsym, epsfig, epic}
  \theoremstyle{plain}
    \newtheorem{theorem}{Theorem}[section]

  \newtheorem{proposition}[theorem]{Proposition}
  \newtheorem{defthm}[theorem]{Definition--Theorem}
  \newtheorem{lemma}[theorem]{Lemma}
  \newtheorem{deflemma}[theorem]{Definition--Lemma}
  
  \newtheorem{conjecture}[theorem]{Conjecture}
  \newtheorem{problem}[theorem]{Problem}
  
\theoremstyle{definition}
  
  \newtheorem{example}[theorem]{Example}

 \theoremstyle{remark}
  \newtheorem{remark}[theorem]{Remark}
\numberwithin{equation}{section}

\def\Ess{{\mathcal{E}}}

\def\Flags{{\rm Flags}}

\newcommand{\cellsize}{15}
\newlength{\cellsz} 
\newsavebox{\cell}
\sbox{\cell}{\begin{picture}(\cellsize,\cellsize)
\put(0,0){\line(1,0){\cellsize}}
\put(0,0){\line(0,1){\cellsize}}
\put(\cellsize,0){\line(0,1){\cellsize}}
\put(0,\cellsize){\line(1,0){\cellsize}}
\end{picture}}
\newcommand\cellify[1]{\def\thearg{#1}\def\nothing{}%
\ifx\thearg\nothing
\vrule width0pt height\cellsz depth0pt\else
\hbox to 0pt{\usebox{\cell} \hss}\fi%
\vbox to \cellsz{
\vss
\hbox to \cellsz{\hss$#1$\hss}
\vss}}
\newcommand\tableau[1]{\vtop{\let\\\cr
\baselineskip -16000pt \lineskiplimit 16000pt \lineskip 0pt
\ialign{&\cellify{##}\cr#1\crcr}}}
\newcommand{\kellsize}{18}
\newlength{\kellsz} 
\newsavebox{\kell}
\sbox{\kell}{\begin{picture}(\kellsize,\kellsize)
\put(0,0){\line(1,0){\kellsize}}
\put(0,0){\line(0,1){\kellsize}}
\put(\kellsize,0){\line(0,1){\kellsize}}
\put(0,\kellsize){\line(1,0){\kellsize}}
\end{picture}}
\newcommand\kellify[1]{\def\thearg{#1}\def\nothing{}%
\ifx\thearg\nothing
\vrule width0pt height\kellsz depth0pt\else
\hbox to 0pt{\usebox{\kell} \hss}\fi%
\vbox to \kellsz{
\vss
\hbox to \kellsz{\hss$#1$\hss}
\vss}}
\newcommand\ktableau[1]{\vtop{\let\\\cr
\baselineskip -16000pt \lineskiplimit 16000pt \lineskip 0pt
\ialign{&\kellify{##}\cr#1\crcr}}}
\font\co=lcircle10

\def\jr{\smash{\raise2pt\hbox{\co \rlap{\rlap{\char'005} \char'007}}
               \raise6pt\hbox{\rlap{\vrule height5pt}}
               \raise2pt\hbox{\rlap{\hskip4pt \vrule height0.4pt depth0pt
                width5.7pt}}
               \raise2pt\hbox{\rlap{\hskip-9.5pt \vrule height.4pt depth0pt
                width6.2pt}}
               \lower6pt\hbox{\rlap{\vrule height4.5pt}}}}
\def\rj{\smash{\raise2pt\hbox{\co \rlap{\rlap{\char'004} \char'006}}
               \raise6pt\hbox{\rlap{\vrule height5pt}}
               \raise2pt\hbox{\rlap{\hskip4pt \vrule height0.4pt depth0pt
                width5.7pt}}
               \raise2pt\hbox{\rlap{\hskip-9.5pt \vrule height.4pt depth0pt
                width6.2pt}}
               \lower6pt\hbox{\rlap{\vrule height4.5pt}}}}
\def\je{\smash{\raise2pt\hbox{\co \rlap{\rlap{\char'005}
                \phantom{\char'007}}}\raise6pt\hbox{\rlap{\vrule height5pt}}
               \raise2pt\hbox{\rlap{\hskip-9.5pt \vrule height.4pt depth0pt
                width6.2pt}}}}
\def\ej{\smash{\raise2pt\hbox{\co \rlap{\rlap{\char'004}\phantom{\char'006}}}
               \raise2pt\hbox{\rlap{\hskip-9.5pt \vrule height.4pt depth0pt
                width6.2pt}}
               \lower6pt\hbox{\rlap{\vrule height4.5pt}}}}
\def\er{\smash{\raise2pt\hbox{\co \rlap{\rlap{\phantom{\char'005}} \char'007}}
               \raise2pt\hbox{\rlap{\hskip4pt \vrule height0.4pt depth0pt
                width5.7pt}}
               \lower6pt\hbox{\rlap{\vrule height4.5pt}}}}
\def\re{\smash{\raise2pt\hbox{\co \rlap{\rlap{\phantom{\char'004}} \char'006}}
               \raise6pt\hbox{\rlap{\vrule height5pt}}
               \raise2pt\hbox{\rlap{\hskip4pt \vrule height0.4pt depth0pt
                width5.7pt}}}}
\def\+{\smash{\lower6pt\hbox{\rlap{\vrule height17pt}}
                \raise2pt
                \hbox{\rlap{\hskip-9pt \vrule height.4pt depth0pt
                width18.7pt}}}}
\def\hor{\smash{\raise2pt\hbox{\rlap{\hskip-9.5pt \vrule height.4pt depth0pt
                width19.2pt}}}}
\def\ver{\smash{\lower6pt\hbox{\rlap{\vrule height17pt}}}}
\def\ho{\smash{\hbox{\rlap{\vrule height5pt}}
                \raise2pt
                \hbox{\rlap{\hskip-9pt \vrule height.4pt depth0pt
                width18.7pt}}}}

\def\textcross{\ \smash{\lower4pt\hbox{\rlap{\hskip4.15pt\vrule height14pt}}
                \raise2.8pt\hbox{\rlap{\hskip-3pt \vrule height.4pt depth0pt
                width14.7pt}}}\hskip12.7pt}
\def\textelbow{\ \hskip.1pt\smash{\raise2.75pt%
                \hbox{\co \hskip 4.15pt\rlap{\rlap{\char'004} \char'006}
                \lower6.8pt\rlap{\vrule height3.5pt}
                \raise3.6pt\rlap{\vrule height3.5pt}}
                \raise2.8pt\hbox{%
                  \rlap{\hskip-7.15pt \vrule height.4pt depth0pt width3.5pt}%
                  \rlap{\hskip4.05pt \vrule height.4pt depth0pt width3.5pt}}}
                \hskip8.7pt}

\begin{document}
\pagestyle{plain}
\title{Some Degenerations of Kazhdan-Lusztig ideals and
multiplicities of Schubert varieties}
\author{Li Li}
\address{Department of Mathematics\\
University of Illinois at Urbana-Champaign\\
Urbana, IL 61801}
\email{llpku@math.uiuc.edu}

\author{Alexander Yong}
\address{Department of Mathematics\\
University of Illinois at Urbana-Champaign\\
Urbana, IL 61801}
\email{ayong@math.uiuc.edu}

\date{January 19, 2010}

\begin{abstract}
We study Hilbert-Samuel multiplicity
for points of Schubert varieties in the complete flag variety, by Gr\"{o}bner degenerations of the Kazhdan-Lusztig
ideal. In the covexillary case, we
give a positive combinatorial rule for multiplicity by
establishing (with a Gr\"{o}bner  basis) a reduced and equidimensional limit whose Stanley-Reisner simplicial
complex is homeomorphic to a shellable ball or sphere. We show that multiplicity counts the
number of facets of this complex. We also obtain a formula for the Hilbert series of the local ring.
In particular, our work gives a multiplicity rule for Grassmannian Schubert varieties,
providing alternative statements and proofs to formulae of [Lakshmibai-Weyman '90], [Rosenthal-Zelevinsky '01],
[Krattenthaler '01], [Kreiman-Lakshmibai '04] and [Woo-Yong '09]. We suggest extensions of our methodology to the general case.
\end{abstract}

\maketitle
\tableofcontents

\section{Introduction}

\subsection{Overview}
Let ${\rm Flags}({\mathbb C}^n)$ denote the variety of complete
flags in ${\mathbb C}^n$. Its Schubert subvarieties $X_w$ are
indexed by permutations $w$ in the symmetric group $S_n$. There has
been substantial interest in understanding the
singularity structure of Schubert varieties. While the singular loci have
been determined, and fundamental properties that hold for all Schubert
varieties have been long established, many mysteries remain about
measures of singularities; see, e.g., \cite{Billey.Lakshmibai, Brion, WYII}.
This paper treats a classical example of such a measure, the
(Hilbert-Samuel) {\bf multiplicity} of a point $p$ in a scheme $X$,
denoted ${\rm mult}_{p}(X)$. This positive integer is the degree of the projectivized
tangent cone $Proj({\rm gr}_{\mathfrak m_p} {\mathcal O}_{p,X})$ as
a subvariety of the projectivized tangent space $Proj(Sym^{\star}
{\mathfrak m_p}/{\mathfrak m}_p^2)$, where $({\mathcal O}_{p,X},
{\mathfrak m_p})$  is the local ring associated to $p\in X$.
Equivalently, if the Hilbert--Samuel polynomial of ${\mathcal
O}_{p,X}$ is $a_d x^d+a_{d-1}x^{d-1}+\ldots +a_0$ ($a_d\neq 0$) then ${\rm
mult}_{p}(X)=d!a_d$. In particular, ${\rm mult}_{p}(X)=1$ if and only if
$X$ is smooth at $p$.

It is an open problem to give a positive combinatorial rule for the
multiplicity of a Schubert variety $X_w$ at its torus fixed points
$e_v\in X_w$ (the problem for arbitrary $p\in X_w$ reduces to this case).
The analogous problem for Grassmannians has been solved; see, e.g.,
\cite{Rosenthal.Zelevinsky, Krattenthaler, Kreiman.Lakshmibai, Kreiman, WYIII}
and the references therein. There has also been related work on multiplicities of
(co)minuscule Grassmannians and for determinantal varieties; a sampling
includes \cite{Lakshmibai.Weyman, Herzog.Trung,  Ghorpade.Raghavan, Ikeda.Naruse, Raghavan.Upadhyay}.

The thesis of this paper is as follows.
A neighbourhood of $e_v\in X_w$ is encoded by the Kazhdan-Lusztig variety ${\mathcal N}_{v,w}$
with explicit coordinates and equations given in \cite{WYII}. We propose to study a choice of
term orders $\prec_{v,w,\pi}$ that depends on $v,w$ and a \emph{shuffling} (total ordering) of variables~$\pi$.
The corresponding Gr\"{o}bner degenerations break
${\mathcal N}_{v,w}$, and its projectivized tangent cone, into an
initial scheme ${\rm init}_{\prec_{v,w,\pi}} {\mathcal N}_{v,w}$ whose reduced
scheme structure is of a union of
coordinate subspaces. By construction, multiplicity is the degree of this monomial ideal.  However,
more seems conjecturally true: first, there exists $\pi$ such that ${\rm
init}_{\prec_{v,w,\pi}} {\mathcal N}_{v,w}$ is both reduced and
equidimensional; and second, one can furthermore choose
$\pi$ so that the corresponding Stanley-Reisner simplicial
complex is homeomorphic to a shellable ball or sphere.
These conjectures assert multiplicity reduces to the combinatorics of counting
the number of facets of a desirable simplicial complex. We label facets
by $\pi$-\emph{shuffled tableaux} that
assign $+$'s to the $n\times n$ grid, using $\pi$ and the corresponding prime
component of the initial ideal.

This paper further formulates the above thesis and collects some
evidence for its efficacy towards the multiplicity problem.

Our main theorems prove the above conjectures for {\bf covexillary Schubert varieties}, i.e.,
those $X_w$ where $w$ avoids the pattern $3412$. We obtain the first multiplicity rule in this case,
which is presently the most general one available in type $A$. Actually, these Schubert varieties have
attracted significant attention in the study of Schubert geometry and
combinatorics; see, e.g., \cite{Lakshmibai.Sandhya, Macdonald, Fulton:Duke92, Lascoux,
Manivel, KMY} and the references therein. For comparison, A.~Lascoux
\cite{Lascoux} studied a different measure of singularities of
Schubert varieties. He gave a combinatorial rule for the
Kazhdan-Lusztig polynomials at singular points of covexillary $X_w$,
extending work of A.~Lascoux and \linebreak M.-P.~Sch\"{u}tzenberger
\cite{LS:KL} for Grassmannian Schubert varieties. Similarly, our
rule also specializes to the Grassmannian case.

For covexillary Schubert varieties, our key observation is that one can pick
$\pi$ (depending on $v,w$) so that the limit scheme is (after $\pi$-shuffling the coordinates and crossing by
affine space) the limit scheme of a matrix Schubert variety \cite{KMY} for
a different covexillary permutation.
We deduce an explicit Gr\"{o}bner basis, with squarefree initial terms, for the Kazhdan-Lusztig ideal under $\prec_{v,w,\pi}$,
extending the Gr\"{o}bner basis theorem of that earlier paper.
The limit is reduced and equidimensional. Using the results of
\cite{KMY}, we prime decompose the initial ideal and show that the $\pi$-shuffled tableaux are in an easy bijection with
flagged semistandard Young tableaux (thus providing some justification for the nomenclature). Hence,
the number of the stated tableaux counts the desired multiplicity, and
as in \cite{WYIII}, a well-known generalization of the Jacobi-Trudi identity yields a simple proof of a determinantal formula.
Also, the Stanley-Reisner complex homeomorphic to a vertex decomposable and hence shellable ball or sphere.
This feature allows us to prove an ``alternating-sign'' formula for a richer invariant than multiplicity,
the Hilbert series of~${\mathcal O}_{e_v,X_w}$.

We remark, that although we work over ${\mathbb C}$, since our Gr\"{o}bner basis involves only coefficients $\pm 1$,
it follows that our formulae are valid over any characteristic. To our best knowledge, independence of characteristic for
multiplicities was not known for general $e_v \in X_w$ (and not even in the covexillary case).

Summarizing, the results in the covexillary case provide some ``proof of concept''
for our thesis.

\subsection{Some related work}
Gr\"{o}bner degeneration has been exploited in a number of related
settings in recent years, and in particular has been applied to the multiplicity problem.
We now discuss some earlier results in type $A$ to provide context for our specific treatment.

V.~Lakshmibai and J.~Weyman \cite{Lakshmibai.Weyman} and V.~Kreiman and~V.~Lakshmibai \cite{Kreiman.Lakshmibai} utilized
standard monomial theory
to determine multiplicity rules for Grassmannians (actually, \cite{Lakshmibai.Weyman} deduces a recursive rule valid for any minuscule $G/P$).

A.~Woo and the second author \cite{WYIII} explain how the
Kazhdan-Lusztig ideals of \cite{WYII} are compatible with the Schubert polynomial
combinatorics of A.~Lascoux and M.-P.~Sch\"{u}tzenberger \cite{Lascoux.Schutzenberger1, Lascoux.Schutzenberger2}.
Moreover, a Gr\"{o}bner basis theorem for arbitrary Kazhdan-Lusztig ideals was obtained,
generalizing work on Schubert determinantal ideals due to \cite{Knutson.Miller}.
The squarefree initial ideal is equidimensional, and the Stanley-Reisner simplicial
complex is homeomorphic to  a shellable ball or sphere; more precisely, it is a \emph{subword complex} as
defined by A.~Knutson and E.~Miller \cite{Knutson.Miller:subword}.  For special cases of Kazhdan-Lusztig
varieties, and choices of $\pi$, the $\pi$-shuffled tableaux are the pipe dreams of S.~Fomin and A.~N.~Kirillov
\cite{Fomin.Kirillov}, and our thesis subsumes the geometric explanation for these pipe dreams
from \cite{Knutson.Miller}. Similar results to \cite{Knutson.Miller}, used in this paper,
were obtained for covexillary Schubert determinantal ideals in \cite{KMY}.

    As an application of \cite{WYIII}, formulae for the multigraded Hilbert series of Kazhdan-Lusztig ideals
were geometrically proved,
where the multigrading comes from the torus action of the invertible diagonal matrices $T\subseteq GL_n$.
While this theorem is actually used in a crucial way in the present paper, in general this Hilbert series does not
help to directly compute multiplicity, because this torus action is not compatible with the dilation action. However,
\emph{if} a Kazhdan-Lusztig ideal happens to already be
homogeneous with respect to the standard grading that assigns each variable degree one, then it is automatic
that it is also the ideal for its projectivized tangent cone, and one can deduce a formula for multiplicity from this
Hilbert series (homogeneity is guaranteed if $w_0v$ is $321$-avoiding; see \cite[pg.~25]{Knutson:frob}). Moreover, it was explained that for the Grassmannian cases, one can always use the trick of \emph{parabolic moving}
to reduce to the homogeneous case. This gives an easy solution to
the Grassmannian multiplicity problem, using Kazhdan-Lusztig ideals. Unfortunately,
even for covexillary Schubert varieties, parabolic moving is ineffective for even some small examples.
The approach of this paper avoids this issue, by using more direct arguments.

While this paper focuses on type $A$, our results should have analogues for other Lie types.
Recent papers of A.~Knutson \cite{Knutson:patches, Knutson:frob} point the way towards coordinates and equations for Kazhdan-Lusztig varieties. His papers also explain how to iteratively degenerate these varieties, although the degenerations he considers are not directly applicable in general to the multiplicity problem, since they do not degenerate the
projectivized tangent cone. Finally, we remark that the notion of covexillary for type $B$ has already been examined in
a paper by S.~Billey and T.~K.~Lam \cite{Billey.Lam}.

\subsection{Organization and summary of results}
In Section~2 we recall necessary preliminaries about flag, Schubert
and Kazhdan-Lusztig varieties. In Section~3 we rigorously formulate the our approach towards
multiplicities. This is encapsulated in our initial theorem (Theorem~\ref{thm:basic}).
In Sections~4--6 we turn to the covexillary setting and state our main theorems.
We begin by stating our Gr\"{o}bner basis theorem (Theorem~\ref{thm:standard}) in Section~4. In Section~5, we state our prime
decomposition theorem (Theorem~\ref{thm:prime}) for the initial ideal of the Kazhdan-Lusztig ideal in terms of flagged tableaux and their
bijectively equivalent pipe dreams.
 Section~6 exploits these results to obtain combinatorial and determinantal rules for the
multiplicity and the Hilbert series of the
projectivized tangent cone (Theorems~\ref{thm:flaggedrule},~\ref{thm:detrule} and~\ref{thm:Hilbert} respectively). Section~7 is
devoted to the proofs of the theorems of Sections~4--6. Finally, in Section~8 we return to the general case
and state our conjectures.

\section{Preliminaries}

We recall some notions about the varieties discussed in this article. Our conventions agree with the ones
used in \cite{WYII, WYIII}.

\subsection{Flag and Schubert varieties}
Let $G=GL_n({\mathbb C})$, $B$ be the Borel subgroup of strictly upper triangular matrices,
$T\subset B$ the maximal torus of diagonal matrices,
and $B_{-}$ the corresponding opposite Borel subgroup of strictly lower triangular matrices.
The {\bf complete flag  variety} is $\Flags({\mathbb C}^n):=G/B$.
The fixed points of $\Flags({\mathbb C}^n)$ under
the left action of $T$ are naturally indexed by the symmetric group $S_n$ thanks to its
role as the Weyl group of $G$; we denote these points $e_v$ for $v\in
S_n$.  One has the {\bf Bruhat decomposition}
\[G/B=\coprod_{w\in S_n} Be_wB/B.\]
The {\bf Schubert cell} is the
$B$-orbit $X_{w}^{\circ}:=Be_wB/B$, and its
closure
$X_w:=\overline{X_{w}^{\circ}}$
is the {\bf Schubert variety}.  It is a subvariety of dimension $\ell(w)$,
where $\ell(w)$ is the length of any reduced word of $w$.  Each Schubert variety
$X_w$ is a union of Schubert cells. The {\bf Bruhat order} is the partial order on
$S_n$ defined by declaring that
$v\leq w$ if $X^\circ_v\subseteq X_w$.

Since every point on $X_w$ is in the $B$-orbit of some $e_v$ (for $v\leq w$ in
Bruhat order), the study of local questions on Schubert
varieties reduces to the case of these fixed points.  An affine neighbourhood
of $e_v$ is given by $v\Omega^{\circ}_{id}$, where in general
$\Omega_{u}^{\circ}:=B_{-}uB/B$
is the {\bf opposite Schubert cell}.  Hence to study $X_w$ locally at
$e_v$ one only needs to understand $X_w\cap v\Omega_{id}^{\circ}$. However, by
\cite[Lemma A.4]{Kazhdan.Lusztig}, one has the isomorphism
\begin{equation}
\label{eqn:KL}
X_{w}\cap v\Omega_{id}^{\circ}\cong (X_w\cap \Omega_{v}^{\circ})\times
{\mathbb A}^{\ell(v)}.
\end{equation}

Hence, we study the (reduced and irreducible)
{\bf Kazhdan--Lusztig variety}
\[{\mathcal N}_{v,w}=X_w\cap
\Omega_{v}^{\circ},\]
harmlessly dropping the factor of affine space.

\subsection{Kazhdan-Lusztig ideals}
We now recall coordinates on $\Omega^\circ_v$, and the {\bf Kazhdan--Lusztig
  ideal} $I_{v,w}$ in these coordinates \cite{WYII}.

Let $M_n$ be the set of all $n\times n$ matrices with entries in ${\mathbb
  C}$, with coordinate ring ${\mathbb C}[{\bf z}]$ where ${\bf
  z}=\{z_{ij}\}_{i,j=1}^{n}$ are the coordinate functions on the entries of
a generic matrix $Z$.  We index the matrix
so that $z_{ij}$ is in the $i$-th row from the {\em bottom} of the matrix and
$j$-th column from the left.  Concretely realizing $G$, $B$, $B_{-}$, and
$T$ as invertible, upper triangular, lower triangular, and diagonal matrices
respectively, as explained in \cite{Fulton:YT}, we can think of the
opposite Schubert cell $\Omega_v^\circ$ as an affine subspace of $M_n$.
Specifically, a matrix is in (our realization of) $\Omega_v^\circ$ if, for all
$i$,
\[z_{n-v(i)+1,i}=1, \mbox{\ and \ }
z_{n-v(i)+1,a}=0  \mbox{\ and $z_{b,i}=0$ for $a>i$ and $b>n-v(i)+1$}.\]
Let ${\bf  z}^{(v)}\subseteq {\bf z}$
denote the remaining unspecialized variables, and
$Z^{(v)}$ the specialized generic matrix representing a generic element of
$\Omega_v^\circ$.

Let $Z_{ab}^{(v)}$
denote the southwest $a\times b$ submatrix of $Z^{(v)}$.
Also let
\[R^w=[r_{ij}^{w}]_{i,j=1}^{n}\]
be the {\bf rank matrix} (which we index
similarly) defined by
$$r_{ij}^{w}=\#\{k\ | \ w(k)\geq n-i+1, k\leq j\}.$$

Define the {\bf Kazhdan--Lusztig ideal}
\[I_{v,w}\subseteq {\mathbb C}[{\bf z}^{(v)}]\cong {\rm Fun}[\Omega_{v}^{\circ}]\]
to be the ideal generated by all of the size $1+r_{ij}^{w}$ minors of
$Z_{ij}^{(v)}$ for all $i$ and $j$.

\subsection{Schubert determinantal ideals}
The {\bf Schubert determinantal ideal}
$I_{w}$ is generated by all size $1+r_{ij}^{w}$ determinants of the southwest
$i\times j$ submatrix $Z_{ij}$ of $Z$, for all $i,j$.
It is known that $I_w$ is generated by the smaller set of {\bf essential
determinants} which is the subset of the above generators coming from only
$(i,j)$ in the essential set of $w$ (we recall the definition of the essential set in
Section~4.1). The {\bf matrix Schubert variety} ${\overline X}_w$ is the (reduced and
irreducible) variety in $M_n$ defined by $I_w$.  Matrix Schubert varieties
were introduced in \cite{Fulton:Duke92}. In fact, matrix Schubert varieties can be realized as
special cases of Kazhdan-Lusztig varieties, as seen in \cite{Fulton:Duke92} and recapitulated in
\cite[Section~2.3]{WYIII}.

\subsection{Torus actions} The action of $T\cong ({\mathbb C}^{\star})^n$ on
${\rm Flags}({\mathbb C}^n)$ induces the {\bf usual action}.
This action is the left action of diagonal matrices on $B$-cosets of $G$
written in our coordinates.  The action rescales rows independently and
rescales columns dependently, as upon rescaling a row one must rescale a
corresponding column to ensure there is a $1$ in position $(n-v(j)+1,j)$ (as
read with our upside-down matrix coordinates).  Applying the usual
convention that the homomorphism picking out the $i$-th diagonal entry is the
weight $t_i$ and writing weights additively, this action gives the matrix
entry at $(i,j)$ the weight $t_{n-i+1}-t_{v(j)}$.  The variable $z_{ij}$ is
the coordinate function on this matrix entry and therefore (the torus action
on the variable) has weight
\[{\rm wt}(z_{ij})=t_{v(j)}-t_{n-i+1}\,.\]
Let us call this the {\bf usual action grading}; it is a fact that this is a positive grading (cf. Section~7.3). The Kazhdan-Lusztig ideal $I_{v,w}$ is homogeneous with respect to the usual action grading, since one can easily check that each defining determinant is homogeneous.

\section{Gr\"{o}bner degeneration  and
multiplicity}

Let $\pi$ be a {\bf shuffling}, i.e., an ordering of the variables of ${\mathbb C}[{\bf z}^{(v)}]$ by reading the
rows of $Z^{(v)}$ from left to right and bottom to top, each of the $\ell(w_0v)!$ orderings of the variables
can be identified with a permutation $\pi$ in the symmetric group $S_{\ell(w_0v)}$.
Let $\prec_{v,w,\pi}'$ be the local term order (i.e., one where $z_{ij}\prec_{v,w,\pi}' 1$)
that favors monomials of lowest total degree first, and then
breaks ties lexicographically according to $\pi$.

Rather than using $\prec_{v,w,\pi}'$ directly, we find it more convenient to study a different
term order $\prec_{v,w,\pi}$ on monomials
in ${\mathbb C}[{\bf z}^{(v)}]$, defined as follows. For each $t_i$, define $\phi(t_i)=n+1-i$.
Define the {\bf non-standard degree} ${\rm deg}$ of $z_{ij}$ to be
$$\aligned {\rm deg}(z_{ij}) & =\phi(t_{v(j)})-\phi(t_{n+1-i})\\
& =n+1-i-v(j).\endaligned$$
Also, define the {\bf standard degree} ${\rm deg}'$ by ${\rm
deg}'(z_{ij})=1$. As usual, extend these definitions to monomials
${\bf m}=c\prod_{ij}z_{ij}^{a_{ij}}$ (where $c\in {\mathbb
C}^{\star}$) by
\[{\rm deg}({\bf
m})=\sum_{ij}a_{ij}{\rm deg}(z_{ij})\]
etc., and where ${\rm deg}(c)={\rm
deg}'(c)=0$. Note that ${\rm deg}({\bf m})$ is a ${\mathbb Z}$-graded coarsening of the usual-action
grading of Section~2.4.

Let ${\bf m}_1$ and ${\bf m}_2$ be two monomials in $\mathbb{C}[\mathbf{z}^{(v)}]$.  Define ${\bf m}_1\prec_{v,w,\pi}{\bf m}_2$ if
\begin{itemize}
\item[(a)] ${\rm deg}({\bf m}_1)<{\rm deg}({\bf m}_2)$, or if
\item[(b)] ${\rm deg}({\bf m}_1)={\rm deg}({\bf m}_2)$ and  ${\bf m}_1\prec_{v,w,\pi}'{\bf m}_2$.
\end{itemize}

The statement of the result below also requires the {\bf Stanley-Reisner
correspondence}. This bijectively associates a
squarefree monomial ideal $I\subseteq {\mathbb C}[z_1,\ldots,z_N]$
with a simplicial complex $\Delta_I$ whose vertex set is
$\{1,2,\ldots,N\}$ and whose faces correspond naturally to monomials
\emph{not} in $I$. Conversely, to each such simplicial complex
$\Delta$, there is an associated ideal $I_{\Delta}\subseteq {\mathbb
C}[z_1,\ldots,z_N]$ and {\bf face ring} ${\mathbb
C}[{\Delta}]={\mathbb C}[z_1,\ldots,z_N]/I_{\Delta}$. Our resource for facts about combinatorial
commutaive algebra is the textbook by E.~Miller and B.~Sturmfels \cite{Miller.Sturmfels}; cf. Section~7.3.

We have:

\begin{theorem}
\label{thm:basic}
Let $\pi\in S_{\ell(w_0 v)}$ be a shuffling for ${\mathbb C}[{\bf z}^{(v)}]$.
Then the following holds:
\begin{itemize}
\item[(I)] $\prec_{v,w,\pi}$ is a global term order (i.e., one where $1\prec_{v,w,\pi} z_{ij}$) such that
if $f\in {\mathbb C}[{\bf z}^{(v)}]$ is homogeneous with respect to usual action grading, then
${\rm init}_{\prec_{v,w,\pi}}(f)={\rm init}_{\prec_{v,w,\pi}'}(f)$.
\item[(II)] ${\rm init}_{\prec_{v,w,\pi}}I_{v,w}={\rm init}_{\prec'_{v,w,\pi}}I_{v,w}={\rm init}_{\prec_{v,w,\pi}} T_{v,w}$, where
\[T_{v,w}=\langle {\hat f}: {\hat f} \mbox{\ is the lowest standard degree component of $f\in I_{v,w}$}\rangle\]
defines the ideal of the projectivized tangent cone of ${\mathcal
N}_{v,w}$; $T_{v,w}$ is homogeneous with respect to both standard and usual action gradings.
\item[(III)] ${\rm mult}_{e_v}(X_w)={\rm degree}(T_{v,w}) = {\rm degree}({\rm init}_{\prec_{v,w,\pi}}I_{v,w})$
\item[(IV)] Under the usual action grading,
the Hilbert series for  ${\mathbb C}[{\bf z}^{(v)}]/I_{v,w}$ equals the Hilbert series of
${\mathbb C}[{\bf z}^{(v)}]/{\rm init}_{\prec_{v,w,\pi}}I_{v,w}$.
\item[(V)] If ${\rm init}_{\prec_{v,w,\pi}} I_{v,w}$ is reduced and equidimensional, then ${\rm mult}_{e_v}(X_w)$
equals the number of irreducible components of ${\rm init}_{\prec_{v,w,\pi}} I_{v,w}$, or alternatively,
equals the number of facets
of the Stanley-Reisner simplicial complex $\Delta_{v,w,\pi}$ associated to ${\rm init}_{\prec_{v,w,\pi}} I_{v,w}$.
\item[(VI)] If in addition to the hypothesis of {\rm (IV)}, $\Delta_{v,w,\pi}$ is homeomorphic to a ball
or sphere, then the ${\mathbb Z}$-graded Hilbert series for ${\mathcal O}_{e_v,X_w}$ is given by
\[\sum_{i\geq 0}\dim({\mathfrak m}_{e_v}^i/{\mathfrak m}_{e_v}^{i+1}) t^i=G_{v,w}(t)/(1-t)^{n\choose 2}\]
where
\[G_{v,w}(t)=\sum_{k\geq 0}(-1)^k(1-t)^{\ell(w_0w)+k}\times \#\mbox{\rm \{interior faces of $\Delta_{v,w,\pi}$ of codimension $k$\}}.\]
\end{itemize}
\end{theorem}

\begin{proof}

For (I), to check that $\prec_{v,w,\pi}$ is a term order, first, we need to show that it is a total ordering on monomials; and second, that it is multiplicative,
meaning that for monomials ${\bf m}_1,{\bf m}_2,{\bf m}_3$, if ${\bf m}_1\prec_{v,w,\pi}{\bf m}_2$
then ${\bf m}_1{\bf m}_3\prec_{v,w,\pi}{\bf m}_2{\bf m}_3$; and third, that it is Artinian, meaning $1\prec_{v,w,\pi}{\bf m}$ for
all nonunit monomials ${\bf m}$.
Clearly $\prec_{v,w,\pi}$ is a total order. It is also straightforward to check that  $\prec_{v,w,\pi}$ is
multiplicative by considering cases (a) and (b) separately.  To see
that $\prec_{v,w,\pi}$ is Artinian, it suffices to show that ${\rm
deg}(1)<{\rm deg}({\bf m})$ for any nonunit monomial ${\bf m}$, hence
$1\prec_{v,w,\pi} {\bf m}$ by (a). Indeed, note that $\prec_{v,w,\pi}$ is a positive weighting on
monomials: if $z_{ij}$ appears in $Z^{(v)}$ then we must have $i<n+1-v(j)$ by construction.
Hence ${\rm deg}(z_{ij})=n+1-v(j)-i>0$. Finally, if $f$ is homogeneous with respect to the usual action grading, then
the comparison of terms of $f$ falls into case (b) of the definition of $\prec_{v,w,\pi}$ and hence we pick
the initial term according to $\prec_{v,w,\pi}'$.

For (II), the equality ${\rm init}_{\prec_{v,w,\pi}}I_{v,w}={\rm init}_{\prec'_{v,w}}I_{v,w}$ follows from (I) and the
fact that $I_{v,w}$ is an homogeneous ideal with respect to the non-standard degree ${\rm deg}$ (cf.~Section 2.4).
The remaining equality and claim about $T_{v,w}$ holds similarly. 

For (III),  the degree of the projectivized tangent cone of
$e_v$ in $X_w$ as a subscheme of the projectivized tangent space
equals the degree of $T_{v,w}$. Hence we have ${\rm mult}_{e_v}(X_w)={\rm
degree} \ T_{v,w}$. That the latter degree equals ${\rm
init}_{\prec_{v,w,\pi}'} I_{v,w}$ is an application of Mora's tangent
cone algorithm \cite{Mora}. Then apply (II).

(IV) holds since the usual action grading is a positive grading
on monomials in ${\mathbb C}[{\bf z}^{(v)}]$
and it is a general fact that Hilbert series for positively graded modules are preserved under Gr\"{o}bner
degeneration, see, e.g., \cite{Miller.Sturmfels}.

For (V), note that by (II) and
(III), ${\rm mult}_{e_v}(X_w)={\rm degree}({\rm
init}_{\prec_{v,w,\pi}}I_{v,w})$. Hence the first claim follows from the hypothesis
and additivity of degrees. The second half of (V) is a standard translation concerning
Stanley-Reisner simplicial complexes.

To prove (VI), we use the following formula established in
\cite[Theorem~4.1]{Knutson.Miller:subword}: if $\Delta$ is a ball or
a sphere and $S$ is its Reisner-Stanley ring, then the $K$-polynomial is given by
$$K(S,t)=\sum_F(-1)^{\dim\Delta-\dim F}\prod_{i\notin
F}(1-t),$$ where sum over all interior faces of $F$ of $\Delta$.
We now apply this formula to
\[S={\mathbb C}[{\bf z}^{(v)}]/{\rm init}_{\prec_{v,w,\pi}}I_{v,w}={\mathbb C}[{\bf z}^{(v)}]/{\rm init}_{\prec_{v,w,\pi}}T_{v,w}.\]
Now,
 $$\#\{i\ |\ i\notin F\} =\#\{\mbox{variables in the ring
 ${\mathbb C}[{\bf z}^{(v)}]$}\}-\dim F-1=\ell(w_0v)-\dim F-1.$$
Using the
$\mathbb{Z}$-grading, the denominator of the ${\mathbb Z}$-graded Hilbert series
for $S$ is
$$(1-t)^{\#\left\{\mbox{variables in the ring
${\mathbb C}[{\bf z}^{(v)}]$}\right\}}=(1-t)^{\ell(w_0v)}.$$

Then
$${\rm Hilb}(S,t)  =\frac{\sum_F(-1)^{\hbox{codim } F} (1-t)^{\ell(w_0v)-\dim F-1}}{(1-t)^{\ell(w_0v)}}=
\frac{\sum_F(-1)^k (1-t)^{{n\choose 2}-\dim F-1}}{(1-t)^{n\choose
2}}$$
where the sum over the interior faces $F$ and where
$k=\dim\Delta-\dim F$ is the codimension of a face $F$.
Since $\dim
\Delta=\ell(w_0 v)-\ell(w_0w)-1=\ell(w)-\ell(v)-1$, we have
${n\choose 2}-\dim F-1=\ell(v)+\ell(w_0w)+k$.

By (\ref{eqn:KL}),  $e_v$ has a neighborhood in $X_w$ that is
isomorphic to ${\mathcal N}_{v,w}\times \mathbb{C}^{\ell(v)}$. Under
this isomorphism, $e_v$ maps to the point $({\bf 0},\vec 0)\in {\mathcal
N}_{v,w}\times \mathbb{C}^{\ell(v)}$, where ${\bf 0}\in {\mathcal N}_{v,w}$
and $\vec 0 \in\mathbb{C}^{\ell(v)}$. So we have
$${\rm Hilb}(\mathcal{O}_{e_v,X_w},t)={\rm Hilb}(\mathcal{O}_{{\bf 0},\mathcal{N}_{v,w}}, t)\cdot\frac{1}{(1-t)^{\ell(v)}}.$$
Meanwhile, the tangent cone of ${\mathcal N}_{v,w}$ at ${\bf 0}$ is ${\rm
Spec}(\mathbb{C}[{\bf z}^{(v)}]/T_{v,w})$, so \[{\rm
Hilb}(\mathcal{O}_{{\bf 0},\mathcal{N}_{v,w}}, t) = {\rm
Hilb}(\mathbb{C}[{\bf z}^{(v)}]/T_{v,w}, t)\] and therefore
$${\rm Hilb}(\mathcal{O}_{e_v,X_w},t)={\rm Hilb}(\mathbb{C}[{\bf
z}^{(v)}]/T_{v,w}, t)\cdot\frac{1}{(1-t)^{\ell(v)}}.$$
Combining these facts, the Hilbert series of the
local ring $\mathcal{O}_{e_v,X_w}$ is
$$\frac{\sum_F(-1)^k (1-t)^{\ell(v)+\ell(w_0w)+k}}{(1-t)^{n\choose
2}}\cdot\frac{1}{(1-t)^{\ell(v)}}=\frac{\sum_F(-1)^k
(1-t)^{\ell(w_0w)+k}}{(1-t)^{n\choose 2}}.$$ Now (VI) immediately follows.
\end{proof}

Since by (II), $T_{v,w}$ is homogeneous with respect to the standard and usual action grading, we remark 
it is not hard to compute the multigraded Hilbert series of ${\mathcal O}_{e_v,X_w}$, for the combined multigrading, with a similar argument as in the
proof of (VI) (replacing the
$\#\mbox{\rm \{interior faces of $\Delta_{v,w,\pi}$ of codimension $k$\}}$ by a Laurent polynomial in $t_1,\ldots,t_n$). 

We need a few more definitions for future reference: We are mainly interested when
${\rm init}_{\prec_{v,w,\pi}}I_{v,w}$ defines a reduced and equidimensional scheme, at which point we
consider its prime decomposition
\[{\rm init}_{\prec_{v,w,\pi}}I_{v,w}=\bigcap J_i\]
where each $J_i=\langle z_{a_1,b_1},\ldots, z_{a_m,b_m}\rangle$. Define the {\bf shuffled generic matrix}
${\widetilde{Z^{(v)}}}$ by starting with $Z^{(v)}$ and reading the rows left to right and bottom to top,
replacing the $k$-th variable in this reading by the $k$ variable of $\pi$. Now define the {\bf $\pi$-shuffled tableau}
associated to $J_i$ to be a filling of the $n\times n$ grid where a $+$ is placed in the positions
of $z_{a_1,b_1},\ldots z_{a_m,b_m}$ of ${\widetilde{Z^{(v)}}}$. These tableaux are closely
related to (and in fact generalize for special choices of $v,w,\pi$) the pipe dreams of \cite{Fomin.Kirillov} as geometrically interpreted by \cite{Knutson.Miller}, and as we will see, they also generalize (flagged) semistandard Young tableaux.

Two remarks about $\pi$-shuffled tableaux are in order. First, strictly speaking, there is no need to shuffle the coordinates
to write down \emph{some} combinatorial object which labels a prime component of ${\rm init}_{\prec_{v,w,\pi}}I_{v,w}$.
However, in the covexillary case, as well in what we surmise about \cite{Knutson.Miller, KMY, WYIII},
it seems that the $\pi$-shuffling converts otherwise weird subsets of $n\times n$ into coherent combinatorics. It is for this
reason that we propose using this transformation in general. Second, in view of the connection to pipe dreams,
it is also plausible to call these objects ``$\pi$-shuffled pipe dreams''. However, at present we do not know of any way
in general to add elbows\ \ \ \ $\jr$ \ \ \ to the positions of $n\times n$ not filled by $+$'s that would generate reasonable strand diagrams
as in \cite{Fomin.Kirillov} that would justify the ``pipe dream'' name (as first introduced in \cite{Knutson.Miller}).

Theorem~\ref{thm:basic} is most likely combinatorially useful
when the limit is reduced and equidimensional. Conjecturally,
there is some term order $\prec_{v,w,\pi}$ such that this is true.
Therefore, multiplicity would be
counted by the inherently combinatorial object $\Delta_{v,w,\pi}$. With this in mind,
the choice of coordinates and equations for the Kazhdan-Lusztig variety is not arbitrary.
Indeed, whether a variety can be Gr\"{o}bner degenerated to a reduced scheme is
embedding dependent. For example, the only two Gr\"{o}bner degenerations of
${\rm Spec}\left({\mathbb C}[x,y]/(x^2-y^2)\right)$ give multiplicity $2$ lines. However, after
the linear change of coordinates $u=x-y, v= x+y$, we arrive at the ${\rm Spec}\left({\mathbb C}[u,v]/(uv)\right)$
which is already a reduced union of coordinate subspaces and hence equal to any of its Gr\"{o}bner limits.

We will discuss the aforementioned conjecture in more specific detail in Section~8.
In the interim, we prove this conjecture in the covexillary case.

\section{A Gr\"{o}bner basis for covexillary Kazhdan-Lusztig ideals}

We now begin our application of Theorem~\ref{thm:basic} to covexillary Schubert
varieties. In this section, we pick $\prec_{v,w,\pi}$
so that the hypotheses of (V) and (VI) of the theorem hold. We then prove a Gr\"{o}bner basis
theorem for this term order that explicates the degeneration.

\subsection{Diagrams, essential sets and covexillary permutations}
We recall some combinatorics of the symmetric group.
Proofs can be found in, e.g., \cite[Chapter~2]{Manivel}.

We give coordinates to the ambient $n\times n$ grid so that $(1,1)$ refers to
the southwest corner, $(n,1)$ refers to the northwest corner, and so on.
To each $w\in S_n$, the {\bf Rothe diagram} $D(w)$ is the following subset of
the $n\times n$ grid:
\begin{equation}
\label{eqn:diagramdef}
D(w)=\{(i,j): i<n-w(j)+1 \mbox{\ and \ } j<w^{-1}(n-i+1)\}.
\end{equation}
Alternatively, this set is described as follows.  Place a dot $\bullet$ in
position $(n-w(j)+1,j)$ for $1\leq j\leq n$. For each dot draw the ``hook''
that extends to the right and above that dot. The boxes that are not in
any hook are the boxes of $D(w)$. We emphasize that
the {\bf graph} of $w$ is given by the
positions of the $\bullet$'s in position $(n-w(j)+1,j)$ (i.e., $w(j)$ units from the top)
because of our indexing conventions.

The {\bf essential set} $\Ess(w)$ can be described as the set of those
boxes which are on the northeast edge of some connected components of $D(w)$.
To be precise,
\begin{equation}
\label{eqn:essentialsetprecisely}
(i,j)\in \Ess(w) \mbox{\ \ \ if \  } (i,j)\in D(w)  \mbox{\ but both \ }
(i+1,j)\not\in D(w) \mbox{\  and \ } (i,j+1)\not\in D(w).
\end{equation}

\begin{proposition}
\label{prop:ess_unique}
A permutation $w\in S_n$ is uniquely determined by its diagram and the restriction of
the rank matrix $R^w=[r_{ij}^w]_{i,j=1}^n$ to its essential set.
\end{proposition}

\begin{proposition}
\label{prop:bruhat}
Permutations $v,w\in S_n$ satisfy $v\leq w$ (in Bruhat order) if and only if
$r_{ij}^{v}\leq r_{ij}^{w}$ for all $(i,j)\in n\times n$.
\end{proposition}

\begin{defthm}
\label{defthm:covex}
The following are equivalent for a permutation $w\in S_n$:
\begin{itemize}
\item[(i)] $w$ is {\bf covexillary}\footnote{In \cite{Manivel} (and other sources) one instead considers \emph{vexillary}
permutations, which are equal to $w_0w$ where $w$ is covexillary and $w_0$ is the longest length
element of $S_n$. The results we use therefore only differ by a change in conventions.};
\item[(ii)] $w$ is $3412$-avoiding, i.e., there do not exist $1\leq i_1<i_2<i_3<i_4\leq n$ such that
$w(i_3)<w(i_4)<w(i_1)<w(i_2)$;
\item[(iii)] the boxes of the essential set of $w$ lie on a piecewise linear curve oriented weakly southeast to northwest; and
\item[(iv)] the diagram $D(w)$, up to a permutation of the rows and the columns gives a Young diagram.
\end{itemize}
\end{defthm}
If (iv) holds, then in fact the Young diagram $\lambda=\lambda(w)$ is unique, and we will refer to this as
the {\bf shape} of the covexillary permutation $w$.

\subsection{The Gr\"{o}bner basis theorem}\label{subsection: The Mora basis
theorem} We now give our central definition, the ordering of
variables $\pi\in S_{\ell(w_0v)}$ that we use in the main results of this section
and the next.

We say that the box $(x,y)$ is {\bf
dominated by $(w,z)$} if $x\le w$ and $y\le z$, i.e.,  if
$(x,y)$ lies in the rectangular region with $(w,z)$ and $(1,1)$ as
its northeast and southwest corners, respectively.

Let $\lambda=\lambda(w)=(\lambda_1\geq \lambda_2\geq\cdots\geq
\lambda_{\ell}>0)$ be as in Section~4.1. For $1\le i\le \ell$, let
\begin{equation}
\label{eqn:deletedspots}
\left\{(\alpha^{(i)}_1,\beta^{(i)}_1),\dots,(\alpha^{(i)}_{k_i},\beta^{(i)}_{k_i})\right\}
\end{equation}
be the coordinates of those 1's in $Z^{(v)}$ that are dominated by
$(b_i, \lambda_i-i+b_i)$. Here $b_i$ is defined as follows:
Let $B(w)$ be the smallest Young diagram (drawn in French notation) with corner in
position $(1,1)$ that contains all of $\Ess(w)$. Then set
\[b_i=\max_{m}\{B(w)_{m}\geq
\lambda(w)_{i}+m-i\}.\]
Observe that
\begin{equation}
\label{eqn:someineq}
b_1\leq b_2\leq \cdots \leq b_{\ell} \mbox{\ and \ } \lambda_1-1+b_1\geq \lambda_2-2+b_2\geq\cdots \geq \lambda_{\ell}-\ell+b_{\ell}.
\end{equation}
(In Section~5.2, $B(w)$ and $b_i$ will be pictorially motivated and utilized.)

By definition,
$\alpha^{(i)}_j=n+1-v(\beta^{(i)}_j)$ for $1\le j\le k_i$. Define
$$\aligned R_i &=\{1,2,\dots,b_i\}\setminus\{\alpha^{(i)}_1,\dots,\alpha^{(i)}_{k_i}\},\\
C_i & =\{1,2,\dots,\lambda_i-i+b_i\}\setminus\{\beta^{(i)}_1,\dots,\beta^{(i)}_{k_i}\},\endaligned$$
and set $R_0=\emptyset$, $R_{\ell+1}=\{1,\dots,n\}$,
$C_0=\{1,\dots,n\}$, $C_{\ell+1}=\emptyset$. From (\ref{eqn:someineq}) we have the filtrations of $\{1,2,\ldots,n\}$:
\[R_0\subseteq R_{1}\subseteq\cdots \subseteq R_{\ell+1} \mbox{\ and \ }  C_{\ell+1}\supseteq C_{\ell}\supseteq\cdots\supseteq C_0.\]
For $0\le i\le \ell$, set
$$R_{i+1}-R_i=\{r^{(i)}_1<\dots<r^{(i)}_{p_i}\}$$  and thus we can define
$\rho\in S_n$ to be the following permutation (written in one-line
notation):
$$\rho:=r^{(0)}_1\cdots r^{(0)}_{p_0} \  r^{(1)}_1 \cdots r^{(1)}_{p_1}
\cdots\  \cdots r^{(\ell)}_1 \cdots r^{(\ell)}_{p_\ell}\in S_n.$$
Similarly, for $0\le i\le \ell$, set
$$C_i-C_{i+1}=\{c^{(i)}_1<\dots<c^{(i)}_{q_i}\}$$
 and let $\chi\in S_n$ be the
following permutation:
$$\chi:=c^{(\ell)}_1 \cdots c^{(\ell)}_{q_\ell} \
c^{(\ell-1)}_1 \cdots c^{(\ell-1)}_{q_{\ell-1}}\ \cdots \ \cdots
c^{(0)}_1 \cdots c^{(0)}_{q_0}\in S_n.$$

 Let ${\widetilde Z}$ be the {\bf shuffled generic matrix} obtained by reordering the rows of the generic matrix
 by $\rho$ and the columns by $\chi$ (cf. Section~3). 

Let $\prec_{v,w,\pi}$ be the term order defined in Section~3, using the
 ordering of
 variables $\pi$ obtained by
 reading the rows of ${\widetilde Z}$ left to right, and from bottom to top.
 Strictly speaking,
 we have defined $\prec_{v,w,\pi}$ as a term order on all monomials in ${\mathbb C}[{\bf z}]$, which we restrict, in the
 obvious way, to one
for monomials in ${\mathbb C}[{\bf z}^{(v)}]$.

The ideal $I_{v,w}$ is known to be generated by a smaller set of generators, i.e., the {\bf essential minors}
which are the $r_{i,j}^w+1$ minors of $Z_{ij}^{(v)}$ for all $(i,j)\in \Ess(w)$, see \cite{WYII} and the references therein.

Our main result is:

\begin{theorem}
\label{thm:standard}
The essential minors of $I_{v,w}$ form a Gr\"{o}bner
basis with respect to the term order $\prec_{v,w,\pi}$.
\end{theorem}

\begin{example}
\label{exa:7531462}
Let $w=7531462$, $v=5123746$ (in one line notation). Then $w$ is covexillary, $\lambda(w)=(4,2,1)$, the Rothe diagram
 $D(w)$ and the matrix of variables $Z^{(v)}$ are given by the following figure.
\begin{figure}[h]\label{figure1}
\begin{picture}(350,105)
\put(-10,50){$D(w)=$}
\put(37.5,0){\makebox[0pt][l]{\framebox(105,105)}}
\put(52.5,15){\line(1,0){60}}
\put(52.5,30){\line(1,0){60}}
\put(52.5,15){\line(0,1){15}}
\put(67.5,15){\line(0,1){15}}
\put(82.5,15){\line(0,1){15}}
\put(97.5,15){\line(0,1){15}}
\put(112.5,15){\line(0,1){15}}
\put(100,20){${\mathfrak e}_1$}
\put(67.5,45){\line(1,0){30}}
\put(67.5,60){\line(1,0){30}}
\put(67.5,45){\line(0,1){15}}
\put(82.5,45){\line(0,1){15}}
\put(97.5,45){\line(0,1){15}}
\put(85,50){${\mathfrak e}_2$}
\put(82.5,75){\line(1,0){15}}
\put(82.5,90){\line(1,0){15}}
\put(82.5,75){\line(0,1){15}}
\put(97.5,75){\line(0,1){15}}
\put(85,80){${\mathfrak e}_3$}
\thicklines
\put(45,7.5){\circle*{4}}
\put(45,7.5){\line(1,0){97.5}}
\put(45,7.5){\line(0,1){97.5}}
\put(60,37.5){\circle*{4}}
\put(60,37.5){\line(1,0){82.5}}
\put(60,37.5){\line(0,1){67.5}}
\put(75,67.5){\circle*{4}}
\put(75,67.5){\line(1,0){67.5}}
\put(75,67.5){\line(0,1){37.5}}
\put(90,97.5){\circle*{4}}
\put(90,97.5){\line(1,0){52.5}}
\put(90,97.5){\line(0,1){7.5}}
\put(105,52.5){\circle*{4}}
\put(105,52.5){\line(1,0){37.5}}
\put(105,52.5){\line(0,1){52.5}}
\put(120,22.5){\circle*{4}}
\put(120,22.5){\line(1,0){22.5}}
\put(120,22.5){\line(0,1){82.5}}
\put(135,82.5){\circle*{4}}
\put(135,82.5){\line(1,0){7.5}}
\put(135,82.5){\line(0,1){22.5}}
\put(200,50){$Z^{(v)}=\left(\begin{matrix}
0 & 1 & 0 & 0 &0 &0 &0\\
0 & z_{62} & 1 & 0 & 0 &0 &0\\
0 & z_{52} & z_{53} & 1 & 0 &0 &0\\
0 & z_{42} & z_{43} & z_{44} & 0 &1 &0\\
1 & 0 & 0 & 0 & 0 &0 &0\\
z_{21} & z_{22} & z_{23} & z_{24} & 0 &z_{26} &1\\
z_{11} & z_{12} & z_{13} & z_{14} & 1 &0 &0\\
\end{matrix}\right)$}
\end{picture}
\end{figure}

The essential set consists of $3$ boxes ${\mathfrak e}_1=(2,5)$, ${\mathfrak e}_2=(4,4)$, ${\mathfrak e}_3=(6,4)$.
The Kazhdan-Lusztig ideal is generated by all
$2\times 2$ minors of $Z^{(v)}_{{\mathfrak e}_1}$, all $3\times 3$ minors of $Z^{(v)}_{{\mathfrak e}_2}$ and all
$4\times 4$ minors of $Z^{(v)}_{{\mathfrak e}_3}$.
\[I_{5123746, 7531462}=\left\langle
\left|\begin{matrix}
z_{21} & z_{22}\\
z_{11} & z_{12}
\end{matrix}\right|, \
\dots,
\left|\begin{matrix}
1 & 0 & 0\\
z_{21} & z_{22} & z_{23}\\
z_{11} & z_{12} & z_{13}
\end{matrix}\right|, \
\dots,
\left|\begin{matrix}
0 & z_{42} & z_{43} & z_{44}\\
1& 0& 0& 0\\
z_{21} & z_{22} & z_{23} &z_{24}\\
z_{11} & z_{12} & z_{13} & z_{14}
\end{matrix}\right|, \
\dots
\right\rangle.
\]

In this example, $R_1=\{2\}$, $R_2=\{1,2,4\}$, $R_3=\{1,2,4\}$, therefore $\rho=2143567\in S_7$. Similarly, $C_1=\{1,2,3,4\}$, $C_2=\{2,3,4\}$, $C_3=\{2\}$, hence $\chi=2341567\in S_7$. Thus we have the shuffled generic matrix
\[{\widetilde Z}=\left(\begin{matrix}
\tilde{z}_{71} & \tilde{z}_{72} & \tilde{z}_{73} & \tilde{z}_{74} & \tilde{z}_{75} & \tilde{z}_{76} & \tilde{z}_{77}\\
\tilde{z}_{61} & \tilde{z}_{62} & \tilde{z}_{63} & \tilde{z}_{64} & \tilde{z}_{65} & \tilde{z}_{66} & \tilde{z}_{67}\\
\tilde{z}_{51} & \tilde{z}_{52} & \tilde{z}_{53} & \tilde{z}_{54} & \tilde{z}_{55} & \tilde{z}_{56} & \tilde{z}_{57}\\
\tilde{z}_{41} & \tilde{z}_{42} & \tilde{z}_{43} & \tilde{z}_{44} & \tilde{z}_{45} & \tilde{z}_{46} & \tilde{z}_{47}\\
\tilde{z}_{31} & \tilde{z}_{32} & \tilde{z}_{33} & \tilde{z}_{34} & \tilde{z}_{35} & \tilde{z}_{36} & \tilde{z}_{37}\\
\tilde{z}_{21} & \tilde{z}_{22} & \tilde{z}_{23} & \tilde{z}_{24} & \tilde{z}_{25} & \tilde{z}_{26} & \tilde{z}_{27}\\
\tilde{z}_{11} & \tilde{z}_{12} & \tilde{z}_{13} & \tilde{z}_{14} & \tilde{z}_{15} & \tilde{z}_{16} & \tilde{z}_{17}\\
\end{matrix}\right)=
\left(\begin{matrix}
z_{72} & z_{73} & z_{74} & z_{71} & z_{75} & z_{76} & z_{77}\\
z_{62} & z_{63} & z_{64} & z_{61} & z_{65} & z_{66} & z_{67}\\
z_{52} & z_{53} & z_{54} & z_{51} & z_{55} & z_{56} & z_{57}\\
z_{32} & z_{33} & z_{34} & z_{31} & z_{35} & z_{36} & z_{37}\\
z_{42} & z_{43} & z_{44} & z_{41} & z_{45} & z_{46} & z_{47}\\
z_{12} & z_{13} & z_{14} & z_{11} & z_{15} & z_{16} & z_{17}\\
z_{22} & z_{23} & z_{24} & z_{21} & z_{25} & z_{26} & z_{27}\\
\end{matrix}\right)\]
satisfying $\tilde{z}_{ij}=z_{\rho(i),\chi(j)}$. Hence
$\prec_{v,w,\pi}$ is defined by the ordering of variables
$${\tilde z}_{11}>{\tilde z}_{12} > \cdots > {\tilde z}_{17}> {\tilde z}_{21}> {\tilde z}_{22}>\cdots>{\tilde z}_{27}>{\tilde z}_{31}>\cdots
$$
$$z_{22} > z_{23} > \cdots > z_{27}> z_{12}> z_{13}>\cdots>z_{17}>z_{42}>\cdots \ ,
$$
by reading the rows left to right, and bottom to top.
Restricting to the variables actually used in $Z^{(v)}$ gives the ordering
$\pi$ to be
\[\pi: \ \ z_{22}>z_{23}>z_{24}>z_{21}>z_{26}>z_{12}>z_{13}>z_{14}>z_{11}>z_{42}>z_{43}>z_{44}>z_{52}>z_{53}>z_{62}.\]
Thus, the given generators form a Gr\"{o}bner basis with respect to $\prec_{v,w,\pi}$ for this choice of~$\pi$.\qed
\end{example}

We record the fact below for future reference. The proof is
immediate from the above definitions:

\begin{lemma}\label{lemma:simple} Let $1\le i\le \ell$ and define
$b_i'=b_i-k_i$, where, as above 
\[k_i=\#\{\mbox{$1$'s dominated by $(b_i,\lambda_i-i+b_i)$}\}.\]
Then
\[\{1,2,\dots,b_i\}\setminus\{\alpha^{(i)}_1,\dots,\alpha^{(i)}_{k_i}\}=R_i=\{r_1,\dots,r_{b_i'}\}\]
and
\[\{1,2,\dots,\lambda_i-i+b_i\}\setminus\{\beta^{(i)}_1,\dots,\beta^{(i)}_{k_i}\}=C_i=\{c_1,\dots,c_{\lambda_i-i+b_i'}\}.\]
\end{lemma}
\section{The prime decomposition theorem}

\subsection{The covexillary permutation $\Theta_{v,w}$}

We now associate to a covexillary $w$ and a permutation
$v\leq w$ a new covexillary permutation $\Theta_{v,w}$.
\begin{deflemma}
\label{deflemma:Theta} Given $v\leq w$ and $w$ covexillary, there is
a unique covexillary permutation $\Theta_{v,w}\in S_n$ such that
$\lambda(w)=\lambda(\Theta_{v,w})$,
and
\[\Ess(\Theta_{v,w})=\{{\mathfrak e}':\mbox{${\mathfrak e}'$ is obtained
by moving an ${\mathfrak e}\in \Ess(w)$ diagonally southwest by
$r_{{\mathfrak e}}^w-r_{{\mathfrak e}}^v$ units\}}\]
where
$r_{{\mathfrak e}'}^{\Theta_{v,w}}=r_{{\mathfrak e}}^w-r_{{\mathfrak e}}^v$,
for each ${\mathfrak e}\in \Ess(w)$.
\end{deflemma}

Although the proof of Definition-Lemma~\ref{deflemma:Theta}
actually describes an iterative algorithm for constructing $\Theta_{v,w}$, we
emphasize that for the main theorems of this section and the next, it is sufficient to know just $\Ess(\Theta_{v,w})$,
which can be handily computed from $v$ and $w$. To be precise, given $D(w)$, one can draw in the $\bullet$'s of (the graph of) $v$. 
Then one moves each box $\mathfrak{e}\in\Ess(w)$ diagonally southwest by the number of $\bullet$'s of
$v$ weakly southwest of it.

The proof is delayed until Section~5.3, where we collect
some related facts.

\begin{example}
Continuing Example~\ref{exa:7531462}, the reader can check that
$\Theta_{5123746,7531462}=4635721$ is the unique permutation satisfying the conditions of
Definition-Lemma~\ref{deflemma:Theta}.
\end{example}

\subsection{From pipe dreams to flagged tableaux}\label{sec:Flagged tableaux and pipe
dreams} Given
\[\lambda=(\lambda_1\geq\lambda_2\geq\ldots\geq\lambda_\ell>0)\]
and a vector of nonnegative integers
\[{\bf
b}=(b_1,\ldots,b_{\ell})\]
define a semistandard Young tableau $T$ of
shape $\lambda$ to be {\bf flagged by ${\bf b}$} if the labels of $T$ in
row $i$ are at most $b_i$.

Associated to each covexillary permutation $w\in S_n$, there is a
flagging ${\bf b}={\bf b}(w)$: As in Section~4.2, consider the smallest {\bf French notation}
Young diagram (i.e., where the $i$-th row from the bottom is of length $\lambda_i$)
$B(w)\subseteq n\times n$ that contains all the boxes of $\Ess(w)$ as
well as the box at $(1,1)$. A {\bf pipe dream}
consists of a placement of $+$'s in a subset of the boxes of
$B(w)$. The {\bf initial pipe dream} for $w$ places $+$'s in each
box of the French Young diagram $\lambda(w)\subseteq B(w)$ with
its southwest corner is at $(1,1)$ (the fact that one has ``$\subseteq$'' is well-known, and follows, e.g., from
the discussion of Section~5.3). Iteratively define all other {\bf pipe
dreams for $w$} by using the following local transformation in any
$2\times 2$ square in $B(w)$:
\[\begin{matrix}
\cdot & \cdot\\
+ & \cdot
\end{matrix}\ \ \mapsto \ \
\begin{matrix}
\cdot & +\\
\cdot & \cdot
\end{matrix}
\]
Each $+$ in the initial pipe dream for $w$ is in obvious one to one correspondence with the box of
$\lambda(w)$ that it sits in. More generally, this extends inductively to every other pipe dream of $w$.
Thus, we can construct a tableau of shape $\lambda(w)$ by recording in each box
the row that its $+$ lies in. Again by induction, using the transformations above, it is easy to verify
that this tableau is semistandard.

\begin{example}
\label{exa:somepipes}
Continuing the previous example, the reader can check that the pipe dreams for $\Theta_{v,w}=4635721$ are given
in Figure~\ref{fig:somepipes} below, where the left pipe dream is the initial pipe dream for
$\Theta_{v,w}$. We have also drawn in $B(\Theta_{v,w})=(4,3,3)$. (Alternatively, starting directly from
$v$ and $w$ one can quickly determine $\Ess(\Theta_{v,w})$ and thus $B(\Theta_{v,w})$, without knowing
$\Theta_{v,w}$ itself, and then write down the pipe dreams.)\qed

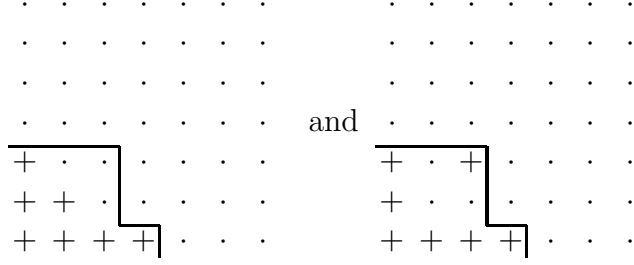
\begin{figure}[h]\setlength{\unitlength}{1.5pt}
\begin{picture}(90,60)
\put(-2,0){$+$}\put(8,0){$+$}\put(18,0){$+$}\put(28,0){$+$}\put(-2,10){$+$}\put(8,10){$+$}\put(-2,20){$+$}
\put(40,0){$\cdot$}\put(50,0){$\cdot$}\put(60,0){$\cdot$}
\put(20,10){$\cdot$} \put(30,10){$\cdot$} \put(40,10){$\cdot$}
\put(50,10){$\cdot$} \put(60,10){$\cdot$} \put(10,20){$\cdot$}
\put(20,20){$\cdot$}\put(30,20){$\cdot$}\put(40,20){$\cdot$}
\put(50,20){$\cdot$}\put(60,20){$\cdot$}
\put(0,30){$\cdot$}\put(10,30){$\cdot$}
\put(20,30){$\cdot$}\put(30,30){$\cdot$}\put(40,30){$\cdot$}
\put(50,30){$\cdot$}\put(60,30){$\cdot$}
\put(0,40){$\cdot$}\put(10,40){$\cdot$}
\put(20,40){$\cdot$}\put(30,40){$\cdot$}\put(40,40){$\cdot$}
\put(50,40){$\cdot$}\put(60,40){$\cdot$}
\put(0,50){$\cdot$}\put(10,50){$\cdot$}
\put(20,50){$\cdot$}\put(30,50){$\cdot$}\put(40,50){$\cdot$}
\put(50,50){$\cdot$}\put(60,50){$\cdot$}
\put(0,60){$\cdot$}\put(10,60){$\cdot$}
\put(20,60){$\cdot$}\put(30,60){$\cdot$}\put(40,60){$\cdot$}
\put(50,60){$\cdot$}\put(60,60){$\cdot$}

\thicklines
\put(-3,26){\line(1,0){28}} \put(25,26){\line(0,-1){20}}
\put(25,6){\line(1,0){10}} \put(35,6){\line(0,-1){8}}
\thinlines
\end{picture}
\begin{picture}(70,65)
\put(-20,30){and}
\put(-2,0){$+$}\put(8,0){$+$}\put(18,0){$+$}\put(28,0){$+$}\put(-2,10){$+$}\put(18,20){$+$}\put(-2,20){$+$}
\put(40,0){$\cdot$}\put(50,0){$\cdot$}\put(60,0){$\cdot$}
\put(20,10){$\cdot$} \put(30,10){$\cdot$} \put(40,10){$\cdot$}
\put(50,10){$\cdot$} \put(60,10){$\cdot$} \put(10,20){$\cdot$}
\put(10,10){$\cdot$}\put(30,20){$\cdot$}\put(40,20){$\cdot$}
\put(50,20){$\cdot$}\put(60,20){$\cdot$}
\put(0,30){$\cdot$}\put(10,30){$\cdot$}
\put(20,30){$\cdot$}\put(30,30){$\cdot$}\put(40,30){$\cdot$}
\put(50,30){$\cdot$}\put(60,30){$\cdot$}
\put(0,40){$\cdot$}\put(10,40){$\cdot$}
\put(20,40){$\cdot$}\put(30,40){$\cdot$}\put(40,40){$\cdot$}
\put(50,40){$\cdot$}\put(60,40){$\cdot$}
\put(0,50){$\cdot$}\put(10,50){$\cdot$}
\put(20,50){$\cdot$}\put(30,50){$\cdot$}\put(40,50){$\cdot$}
\put(50,50){$\cdot$}\put(60,50){$\cdot$}
\put(0,60){$\cdot$}\put(10,60){$\cdot$}
\put(20,60){$\cdot$}\put(30,60){$\cdot$}\put(40,60){$\cdot$}
\put(50,60){$\cdot$}\put(60,60){$\cdot$}

\thicklines
\put(-3,26){\line(1,0){28}}
\put(25,26){\line(0,-1){20}} \put(25,6){\line(1,0){10}}
\put(35,6){\line(0,-1){8}}
\thinlines
\end{picture}
\caption{\label{fig:somepipes} Pipe dreams for $\Theta_{v,w}=4635721$, and $B(\Theta_{v,w})$}
\end{figure}

\end{example}

Not every semistandard tableau of shape $\lambda(w)$ can be obtained
this way. The maximum entry of row $i$ of such a tableau $T$ is
bounded from above by how far north the rightmost $+$ in the $i$-th
row of the starting pipe dream can travel diagonally (not taking into account any other $+$'s)
and remain inside $B(w)$. Let $b_i$ denote this row number.
Actually, this gives the same $b_i$ as defined in Section~4.2, which we recall:
\[b_i=\max_{m}\{B(w)_{m}\geq
\lambda(w)_{i}+m-i\}.\]

\begin{example}
\label{exa:correspondtab}
The corresponding tableaux to the above pipe dreams are:
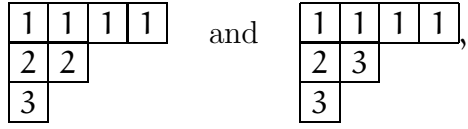
\[\tableau{
{1 }&{ 1 }&{ 1 }&{ 1 }\\
{2 }&{ 2 }\\
{3 }\\
}\quad \textrm{ and }\quad
\tableau{
{1 }&{ 1 }&{ 1 }&{ 1 }\\
{2 }&{ 3 }\\
{3}\\
},\] and here ${\bf b}(\Theta_{v,w})=(1,3,3)$. \qed

\end{example}

\begin{theorem}
\label{thm:prime}
We have
\[{\rm init}_{\prec_{v,w,\pi}} \ I_{v,w}=\bigcap_{\mathcal P} \langle {\widetilde z}_{ij} \ : (i,j)\in {\mathcal P}\rangle\]
where $\tilde{z}_{ij}=z_{\rho(i),\chi(j)}$ (cf. Section~4.2 and Example~\ref{exa:7531462}). Here
the intersection is over all pipe dreams for $\Theta_{v,w}$.

The associated Stanley
Reisner complex $\Delta_{v,w,\pi}$
is homeomorphic to a vertex decomposable ball or sphere. In particular, the
limit defines an equidimensional scheme.

The irreducible components, or equivalently, the facets of $\Delta_{v,w,\pi}$ are in bijection with
semistandard Young tableaux of shape $\lambda(w)$ and
flagged by ${\bf b}(\Theta_{v,w})$.
\end{theorem}

\begin{example}
\label{exa:bizarre}
We have the following prime decomposition
$$\aligned{\rm init}_{\prec_{v,w,\pi}} I_{5123746, 7531462}&=\langle
 z_{12}, z_{21},z_{22}, z_{23}, z_{24}, z_{42}, z_{13}z_{44}  \rangle\\
 &=\langle
z_{12},z_{13},z_{21},z_{22},z_{23},z_{24},z_{42}\rangle\cap \langle
z_{12},z_{21},z_{22},z_{23},z_{24},z_{42},z_{44}\rangle.\endaligned$$

We can associate a $\pi$-shuffled tableaux to each component by placing a $+$ in the position
 of $z_{ab}$ in the shuffled generic matrix ${\widetilde Z}$ whenever $z_{ab}$ appears as a generator of the prime
 ideal for that component (and $\cdot$'s everywhere else).
The result are precisely the pipe dreams given in
Example~\ref{exa:somepipes}, which are themselves in an easy bijection with the semistandard Young tableaux
of Example~\ref{exa:correspondtab}. This accounts for the use of ${\widetilde z}_{ij}$ in
Theorem~\ref{thm:prime}, and provides some rationale for our introduction of $\pi$-shuffled tableaux in general in Section~3.
\qed
\end{example}

\subsection{Proof of Definition-Lemma~\ref{deflemma:Theta} and some properties of ${\bf b}$ and $B(w)$}
Suppose $w$ is covexillary and
\[{\mathfrak e}=(i_0,j_0)\in \Ess(w) \mbox{\ where $r_{i_0,j_0}^w>0$.}\]
Define the {\bf transitioned permutation}
$w'$ as follows. Let $(i_1,j_1)$ be the northeast most dot in $D(w)$ that is dominated by $(i_0,j_0)$.
Such a dot exists because of the assumption $r_{i_0,j_0}^w>0$. By the condition that $w$ is covexillary,
there is at one such choice. Let $(i_2,j_2)$ be the dot that is in the same column as $(i_0,j_0)$, and $(i_3,j_3)$ be the dot that is in
the same row as $(i_0,j_0)$. Hence
\[i_2=i_0+1, j_2=j_0 \mbox{ \ and \ } i_3=i_0, j_3=j_0+1.\]
Then define $w'$ by letting
\begin{equation}
\label{eqn:transitioned}
\aligned
& w'(j_1)=n+1-i_3=n+1-i_0,\\
&w'(j_2)=n+1-i_1=w(j_1),\\
&w'(j_3)=n+1-i_2=w(j_0),\\
& \mbox{and $w'(j)=w(j)$ for $j\neq j_1,j_2,j_3$}.\endaligned
\end{equation}
The figure below illustrates this description of $w'$:

\begin{figure}[h]\label{figure1}
\begin{picture}(350,115)
\put(-40,100){$D(w)$}
\put(0,0){\makebox[0pt][l]{\framebox(140,120)}}
\put(20,20){\line(1,0){60}}
\put(20,40){\line(1,0){60}}
\put(20,60){\line(1,0){40}}
\put(20,20){\line(0,1){40}}
\put(40,20){\line(0,1){40}}
\put(60,20){\line(0,1){40}}
\put(80,20){\line(0,1){20}}
\put(10,70){$(i_0,j_0)$}
\put(35,65){\vector(1,-1){10}}
\put(45,45){${\mathfrak e}$}
\put(-38,20){$(i_1,j_1)$}
\put(-3,20){\vector(1,-1){10}}
\thicklines
\put(10,10){\circle*{4}}
\put(10,10){\line(1,0){130}}
\put(10,10){\line(0,1){110}}
\put(10,95){$(i_2,j_2)$}
\thinlines
\put(35,90){\vector(1,-1){10}}
\thicklines
\put(50,75){\circle*{4}}
\put(50,75){\line(1,0){90}}
\put(50,75){\line(0,1){45}}
\put(72,55){$(i_3,j_3)$}
\put(70,50){\circle*{4}}
\put(70,50){\line(1,0){70}}
\put(70,50){\line(0,1){70}}
\end{picture}
\begin{picture}(600,0)(-300,-15)
\put(-40,100){$D(w')$}
\put(0,0){\makebox[0pt][l]{\framebox(140,120)}}
\put(0,20){\line(1,0){40}}
\put(0,40){\line(1,0){40}}
\put(20,0){\line(0,1){40}}
\put(40,0){\line(0,1){40}}
\put(60,40){\line(1,0){20}}
\put(60,20){\line(1,0){20}}
\put(60,20){\line(0,1){20}}
\put(80,20){\line(0,1){20}}
\put(-80,25){$(i_0-1,j_0-1)$}
\put(-5,30){\vector(1,0){30}}
\put(25,25){${\mathfrak e}'$}
\put(-50,50){$(i_3,j_1)$}
\put(-15,50){\vector(1,0){20}}
\thicklines
\put(50,10){\circle*{4}}
\put(50,10){\line(1,0){90}}
\put(50,10){\line(0,1){110}}
\put(68,-12){$(i_1,j_2)$}
\thinlines
\put(65,-10){\vector(-1,1){15}}
\thicklines
\put(70,75){\circle*{4}}
\put(70,75){\line(1,0){70}}
\put(70,75){\line(0,1){45}}
\put(70,80){$(i_2,j_3)$}
\put(10,50){\circle*{4}}
\put(10,50){\line(1,0){130}}
\put(10,50){\line(0,1){70}}
\end{picture}
\caption{Going from $D(w)$ to $D(w')$ in Definition-Lemma~\ref{deflemma:Theta}}
\end{figure}
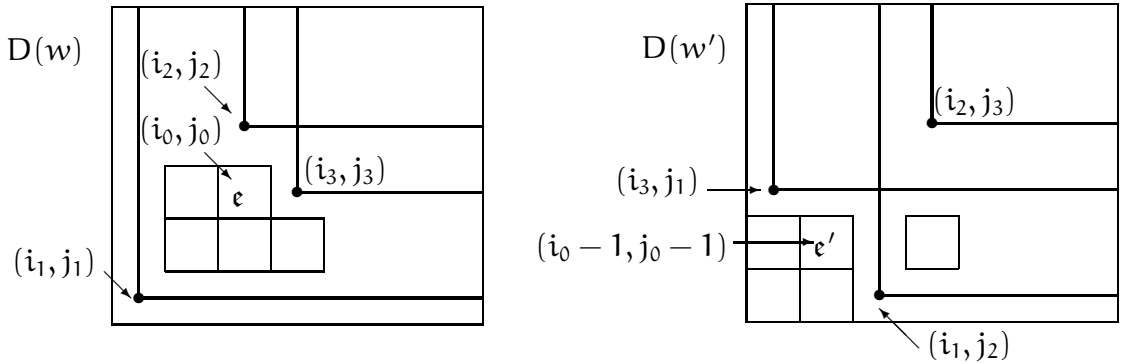
\smallskip

The proof of Definition-Lemma~\ref{deflemma:Theta} is based on the following fact, whose proof is
straightforward and omitted (cf. \cite[Lemma~3.5]{KMY}).

\begin{lemma}
\label{lemma:transitioned}
Let $w$ be covexillary, $(i_0,j_0)\in \Ess(w)$ with $r_{i_0,j_0}^w>0$. Then the transitioned
permutation $w'$ defined by (\ref{eqn:transitioned}) has the following properties:
\begin{itemize}
\item[(i)] $w'$ is covexillary;
\item[(ii)] $\lambda(w)=\lambda(w')$;
\item[(iii)] $\Ess(w')=\left(\Ess(w)\setminus \{(i_0,j_0)\}\right)\cup \{(i_0 -1,j_0-1)\}$; in particular $w$ and $w'$ have the same number of essential set boxes; and
\item[(iv)] $r_{i_0 -1,j_0-1}^{w'}=r_{i_0,j_0}^w -1$ and $r_{{\mathfrak e}}^{w'}=r_{{\mathfrak e}}^w$ for the remaining
(common) essential set boxes~${\mathfrak e}$.
\end{itemize}
\end{lemma}

\noindent
\emph{Proof of Definition-Lemma~\ref{deflemma:Theta}:}
We algorithmically construct the covexillary
permutation $\Theta_{v,w}$ with the
stated essential set and rank conditions.
(Once achieved,
Proposition~\ref{prop:ess_unique} implies that the permutation is unique, and hence
$\Theta_{v,w}$ is well defined.)

Attach to each essential box $\mathfrak{e}$ the
nonnegative integer
\[f(\mathfrak{e})=r_{{\mathfrak e}}^w-r_{{\mathfrak e}}^v,\]
thought of as indicating how many steps the box $\mathfrak{e}$ should be moved in the southwest direction.
Note that $f({\mathfrak e})\geq 0$ follows from Proposition~\ref{prop:bruhat}. Let
\[k=\sum_{{\mathfrak e}\in\Ess(w)} f({\mathfrak e}).\]
We repeat the following process,
which decreases exactly one of the $f({\mathfrak e})$ by $1$, giving another (intermediate) covexillary permutation.
We terminate when all the $f({\mathfrak e})$'s become $0$ at which point we output $\Theta_{v,w}$.

Define
$$\aligned &i_0=\max\{i\;|\; (i,j) \textrm{ is an essential box and } f((i,j))>0 \},\\
&j_0=\max\{j\;|\; (i_0,j) \textrm{ is an essential box and } f((i_0,j))>0\}.\endaligned$$
Hence $(i_0,j_0)$ gives the coordinates of the northmost then eastmost essential set box that still needs to be
moved (in particular, $r_{i_0,j_0}^{w}>0$). Then set $w'$ to be the transitioned permutation for $w$.

By Lemma~\ref{lemma:transitioned}, $w'$ is covexillary, and $\Ess(w')$ and $\Ess(w)$ are the same except
that $(i_0,j_0)\in \Ess(w)$ has now moved to $(i_0-1,j_0-1)\in \Ess(w')$. Attach to
$(i_0-1,j_0-1)$ the integer $f(i_0,j_0)-1$ and keep the attached integers unchanged for other
essential boxes.

Repeat the algorithm for $w'$. In view of Lemma~\ref{lemma:transitioned}(iv) it follows that we can
do this process $k$ steps. We obtain a permutation and
name it $\Theta_{v,w}$. That $\Theta_{v,w}$ has the desired properties follows from the construction and
inductively applying parts (i), (ii) and (iii) of Lemma~\ref{lemma:transitioned}.\qed

\begin{example}
We now illustrate the algorithm described
in the proof of Definition-Lemma~\ref{deflemma:Theta} by computing
$\Theta_{5123746,7531462}=4635721$ in steps:

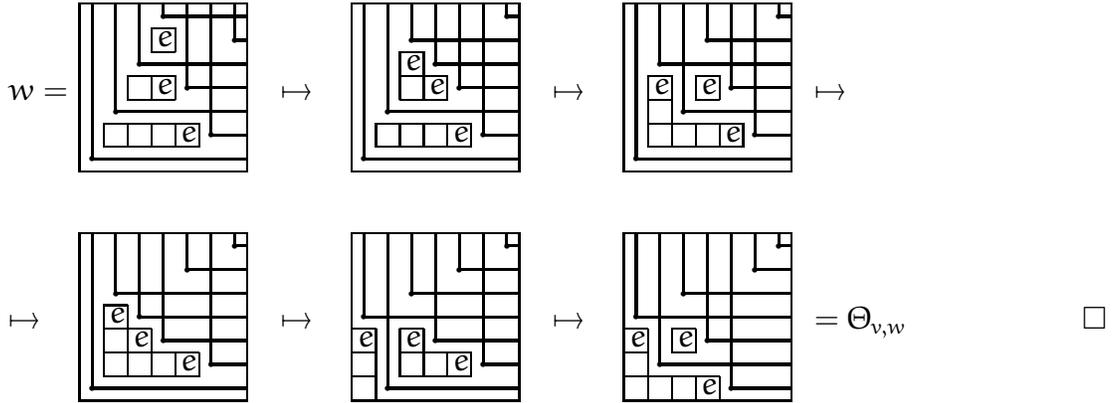
\begin{figure}[h]\setlength{\unitlength}{.45pt}
\begin{picture}(220,140)
\put(-60,60){$w=$} \put(0,0){\makebox[0pt][l]{\framebox(140,140)}}
\put(20,20){\line(1,0){80}} \put(20,40){\line(1,0){80}}
\put(40,60){\line(1,0){40}} \put(40,80){\line(1,0){40}}
\put(60,100){\line(1,0){20}} \put(60,120){\line(1,0){20}}
\put(20,20){\line(0,1){20}} \put(40,20){\line(0,1){20}}
\put(60,20){\line(0,1){20}} \put(80,20){\line(0,1){20}}
\put(100,20){\line(0,1){20}} \put(40,60){\line(0,1){20}}
\put(60,60){\line(0,1){20}} \put(80,60){\line(0,1){20}}
\put(60,100){\line(0,1){20}} \put(80,100){\line(0,1){20}}
\put(85,25){$e$} \put(65,65){$e$} \put(65,105){$e$}
\thicklines
\put(10,10){\circle*{4}}
\put(10,10){\line(1,0){130}}
\put(10,10){\line(0,1){130}}
\put(30,50){\circle*{4}}
\put(30,50){\line(1,0){110}}
\put(30,50){\line(0,1){90}}
\put(50,90){\circle*{4}}
\put(50,90){\line(1,0){90}}
\put(50,90){\line(0,1){50}}
\put(70,130){\circle*{4}}
\put(70,130){\line(1,0){70}}
\put(70,130){\line(0,1){10}}
\put(90,70){\circle*{4}}
\put(90,70){\line(1,0){50}}
\put(90,70){\line(0,1){70}}
\put(110,30){\circle*{4}}
\put(110,30){\line(1,0){30}}
\put(110,30){\line(0,1){110}}
\put(130,110){\circle*{4}}
\put(130,110){\line(1,0){10}}
\put(130,110){\line(0,1){30}}
\end{picture}
\begin{picture}(220,140)
\put(-60,60){$\mapsto$}
\put(0,0){\makebox[0pt][l]{\framebox(140,140)}}
\put(20,20){\line(1,0){80}} \put(20,40){\line(1,0){80}}
\put(40,60){\line(1,0){40}} \put(40,80){\line(1,0){40}}
\put(40,80){\line(1,0){20}} \put(40,100){\line(1,0){20}}
\put(20,20){\line(0,1){20}} \put(40,20){\line(0,1){20}}
\put(60,20){\line(0,1){20}} \put(80,20){\line(0,1){20}}
\put(100,20){\line(0,1){20}} \put(40,60){\line(0,1){20}}
\put(60,60){\line(0,1){20}} \put(80,60){\line(0,1){20}}
\put(40,80){\line(0,1){20}} \put(60,80){\line(0,1){20}}
\put(85,25){$e$} \put(65,65){$e$} \put(45,85){$e$}
\thicklines
\put(10,10){\circle*{4}}
\put(10,10){\line(1,0){130}}
\put(10,10){\line(0,1){130}}
\put(30,50){\circle*{4}}
\put(30,50){\line(1,0){110}}
\put(30,50){\line(0,1){90}}
\put(70,90){\circle*{4}}
\put(70,90){\line(1,0){70}}
\put(70,90){\line(0,1){50}}
\put(130,130){\circle*{4}}
\put(130,130){\line(1,0){10}}
\put(130,130){\line(0,1){10}}
\put(90,70){\circle*{4}}
\put(90,70){\line(1,0){50}}
\put(90,70){\line(0,1){70}}
\put(110,30){\circle*{4}}
\put(110,30){\line(1,0){30}}
\put(110,30){\line(0,1){110}}
\put(50,110){\circle*{4}}
\put(50,110){\line(1,0){90}}
\put(50,110){\line(0,1){30}}
\end{picture}
\begin{picture}(220,140)
\put(-60,60){$\mapsto$}
\put(0,0){\makebox[0pt][l]{\framebox(140,140)}}
\put(20,20){\line(1,0){80}} \put(20,40){\line(1,0){80}}
\put(20,60){\line(1,0){20}} \put(60,60){\line(1,0){20}}
\put(20,80){\line(1,0){20}} \put(60,80){\line(1,0){20}}
\put(20,20){\line(0,1){20}} \put(40,20){\line(0,1){20}}
\put(60,20){\line(0,1){20}} \put(80,20){\line(0,1){20}}
\put(100,20){\line(0,1){20}} \put(20,60){\line(0,1){20}}
\put(40,60){\line(0,1){20}} \put(60,60){\line(0,1){20}}
\put(80,60){\line(0,1){20}} \put(20,40){\line(0,1){20}}
\put(40,40){\line(0,1){20}} \put(85,25){$e$} \put(65,65){$e$}
\put(25,65){$e$}
\thicklines \put(10,10){\circle*{4}} \put(10,10){\line(1,0){130}}
\put(10,10){\line(0,1){130}} \put(50,50){\circle*{4}}
\put(50,50){\line(1,0){90}} \put(50,50){\line(0,1){90}}
\put(30,90){\circle*{4}} \put(30,90){\line(1,0){110}}
\put(30,90){\line(0,1){50}} \put(130,130){\circle*{4}}
\put(130,130){\line(1,0){10}} \put(130,130){\line(0,1){10}}
\put(90,70){\circle*{4}} \put(90,70){\line(1,0){50}}
\put(90,70){\line(0,1){70}} \put(110,30){\circle*{4}}
\put(110,30){\line(1,0){30}} \put(110,30){\line(0,1){110}}
\put(70,110){\circle*{4}} \put(70,110){\line(1,0){70}}
\put(70,110){\line(0,1){30}} \put(160,60){$\mapsto$}
\end{picture}
\end{figure}

\begin{figure}[h]\setlength{\unitlength}{.45pt}
\begin{picture}(220,140)
\put(-60,60){$\mapsto$}
\put(0,0){\makebox[0pt][l]{\framebox(140,140)}}
\put(20,20){\line(1,0){80}} \put(20,40){\line(1,0){80}}
\put(20,60){\line(1,0){20}} \put(20,80){\line(1,0){20}}
\put(40,60){\line(1,0){20}} \put(20,20){\line(0,1){20}}
\put(40,20){\line(0,1){20}} \put(60,20){\line(0,1){20}}
\put(80,20){\line(0,1){20}} \put(100,20){\line(0,1){20}}
\put(20,60){\line(0,1){20}} \put(40,60){\line(0,1){20}}
\put(20,40){\line(0,1){20}} \put(40,40){\line(0,1){20}}
\put(60,40){\line(0,1){20}} \put(85,25){$e$} \put(45,45){$e$}
\put(25,65){$e$}
\thicklines
\put(10,10){\circle*{4}}
\put(10,10){\line(1,0){130}}
\put(10,10){\line(0,1){130}}
\put(70,50){\circle*{4}}
\put(70,50){\line(1,0){70}}
\put(70,50){\line(0,1){90}}
\put(30,90){\circle*{4}}
\put(30,90){\line(1,0){110}}
\put(30,90){\line(0,1){50}}
\put(130,130){\circle*{4}}
\put(130,130){\line(1,0){10}}
\put(130,130){\line(0,1){10}}
\put(50,70){\circle*{4}}
\put(50,70){\line(1,0){90}}
\put(50,70){\line(0,1){70}}
\put(110,30){\circle*{4}}
\put(110,30){\line(1,0){30}}
\put(110,30){\line(0,1){110}}
\put(90,110){\circle*{4}}
\put(90,110){\line(1,0){50}}
\put(90,110){\line(0,1){30}}
\end{picture}
\begin{picture}(220,140)
\put(-60,60){$\mapsto$}
\put(0,0){\makebox[0pt][l]{\framebox(140,140)}}
\put(40,20){\line(1,0){60}} \put(40,40){\line(1,0){60}}
\put(40,60){\line(1,0){20}} \put(0,20){\line(1,0){20}}
\put(0,40){\line(1,0){20}} \put(0,60){\line(1,0){20}}
\put(20,0){\line(0,1){60}} \put(40,20){\line(0,1){20}}
\put(60,20){\line(0,1){20}} \put(80,20){\line(0,1){20}}
\put(100,20){\line(0,1){20}} \put(40,40){\line(0,1){20}}
\put(60,40){\line(0,1){20}} \put(85,25){$e$} \put(45,45){$e$}
\put(05,45){$e$}
\thicklines
\put(30,10){\circle*{4}}
\put(30,10){\line(1,0){110}}
\put(30,10){\line(0,1){130}}
\put(70,50){\circle*{4}}
\put(70,50){\line(1,0){70}}
\put(70,50){\line(0,1){90}}
\put(50,90){\circle*{4}}
\put(50,90){\line(1,0){90}}
\put(50,90){\line(0,1){50}}
\put(130,130){\circle*{4}}
\put(130,130){\line(1,0){10}}
\put(130,130){\line(0,1){10}}
\put(10,70){\circle*{4}}
\put(10,70){\line(1,0){130}}
\put(10,70){\line(0,1){70}}
\put(110,30){\circle*{4}}
\put(110,30){\line(1,0){30}}
\put(110,30){\line(0,1){110}}
\put(90,110){\circle*{4}}
\put(90,110){\line(1,0){50}}
\put(90,110){\line(0,1){30}}
\end{picture}
\begin{picture}(220,140)
\put(-60,60){$\mapsto$}
\put(0,0){\makebox[0pt][l]{\framebox(140,140)}}
\put(0,20){\line(1,0){80}} \put(0,40){\line(1,0){20}}
\put(0,60){\line(1,0){20}} \put(0,20){\line(1,0){20}}
\put(40,40){\line(1,0){20}} \put(40,60){\line(1,0){20}}
\put(20,0){\line(0,1){60}} \put(40,0){\line(0,1){20}}
\put(60,0){\line(0,1){20}} \put(80,0){\line(0,1){20}}
\put(40,40){\line(0,1){20}} \put(60,40){\line(0,1){20}}
\put(65,5){$e$} \put(45,45){$e$} \put(5,45){$e$}
\thicklines \put(90,10){\circle*{4}} \put(90,10){\line(1,0){50}}
\put(90,10){\line(0,1){130}} \put(70,50){\circle*{4}}
\put(70,50){\line(1,0){70}} \put(70,50){\line(0,1){90}}
\put(50,90){\circle*{4}} \put(50,90){\line(1,0){90}}
\put(50,90){\line(0,1){50}} \put(130,130){\circle*{4}}
\put(130,130){\line(1,0){10}} \put(130,130){\line(0,1){10}}
\put(10,70){\circle*{4}} \put(10,70){\line(1,0){130}}
\put(10,70){\line(0,1){70}} \put(30,30){\circle*{4}}
\put(30,30){\line(1,0){110}} \put(30,30){\line(0,1){110}}
\put(110,110){\circle*{4}} \put(110,110){\line(1,0){30}}
\put(110,110){\line(0,1){30}}\put(160,60){$=\Theta_{v,w} \quad \quad \quad \quad \ \ \qed$}
\end{picture}
\end{figure}
\end{example}

In Section~7.3 we will need two properties of ${\bf b}(w)$, whose proofs also follow from
Lemma~\ref{lemma:transitioned}:

\begin{lemma}
\label{lemma:bfact1}
Suppose $w$ is covexillary and $\lambda=\lambda(w)=(\lambda_1\geq\lambda_2\geq\cdots\geq \lambda_{\ell}>0)$.
Furthermore, set
\[\{i_1<i_2<\cdots<i_m\}=\{i \ |\ 1\leq i\leq \ell, \lambda_i>\lambda_{i+1}\}.\]
Then
\[\Ess(w)=\{(b_i,\lambda_i-i+b_i)\;|\;i=i_1,\dots,i_m\}.\]
In other words, there is a one to one bijection between $\Ess(w)$ and the righthand corners of $\lambda$
(drawn in French notation).
\end{lemma}
\begin{proof}
Consider a sequence
\[w=w^{(0)}\mapsto w^{(1)}\mapsto \cdots \mapsto w^{(M)}\]
where
$w^{(i+1)}$ is the transitioned permutation of $w^{(i)}$, and $w^{(M)}$ has the property
that the rank of each of its essential set boxes is $0$ (this permutation is often known as
``dominant''). The diagram of $w^{(M)}$ is a Young diagram (drawn in French notation), with
its southwest corner at $(1,1)$. By Lemma~\ref{lemma:transitioned}(ii), this Young diagram must be $\lambda(w)$.

The essential set boxes of $w^{(M)}$ are precisely boxes at the end of the rows $\{i_1<i_2<\cdots<i_m\}$.
By Lemma~\ref{lemma:transitioned}(iii), it follows that each such ${\mathfrak e}'\in \Ess(w^{(M)})$ (say in
row $i\in  \{i_1<i_2<\cdots<i_m\}$
corresponds one to one with ${\mathfrak e}\in\Ess(w)$ that lives on the same southwest-northeast
diagonal. However, by the definition of $b_i$, ${\mathfrak e}$ must coordinates $(b_i,\lambda_i-i+b_i)$, since
both are the extremal box of $B(w)$ on the said diagonal.
\end{proof}

\begin{lemma}
\label{lemma:bfact2}
Under the same assumptions as Lemma~\ref{lemma:bfact1}, we have
\[b_i=\max(b_{i_{k+1}}-i_{k+1}+i,b_{i_k})\]
(define $b_{i_0}=b_0=0$).
\end{lemma}
\begin{proof}
To see this, consider the following picture:

\begin{figure}[h]\label{figure1}
\begin{picture}(150,110)
\put(20,100){(a)} \put(-10,70){\line(1,0){30}}
\put(20,70){\line(0,-1){40}} \put(20,30){\line(1,0){40}}
\put(60,30){\line(0,-1){30}}
\put(10,70){\line(0,-1){10}} \put(10,60){\line(1,0){10}}
\put(20,80){
$(b_{i_{k+1}},\lambda_{i_{k+1}}-i_{k+1}+b_{i_{k+1}})$}
\put(25,75){\line(-1,-1){10}}
\put(10,50){\line(1,0){10}} \put(10,50){\line(0,-1){10}}
\put(10,40){\line(1,0){10}}
\put(23,58){$(b_{i},\lambda_i-i+b_i)$}
\put(25,55){\line(-1,-1){10}}
\put(50,20){\line(0,1){10}} \put(50,20){\line(1,0){10}}
\put(50,40){$(b_{i_{k}},\lambda_{i_{k}}\!-i_{k}\!+\!b_{i_{k}})$}
\put(66,35){\line(-1,-1){10}}
\put(-10,0){\line(0,1){70}} \put(-10,0){\line(1,0){70}}
\put(-5,10){$B(w)$}
\end{picture}
\begin{picture}(60,110)
\put(20,100){(b)} \put(-10,70){\line(1,0){30}}
\put(20,70){\line(0,-1){40}} \put(20,30){\line(1,0){40}}
\put(60,30){\line(0,-1){30}}
\put(10,70){\line(0,-1){10}} \put(10,60){\line(1,0){10}}
\put(20,80){
$(b_{i_{k+1}},\lambda_{i_{k+1}}-i_{k+1}+b_{i_{k+1}})$}
\put(25,75){\line(-1,-1){10}}
\put(30,20){\line(1,0){10}} \put(30,20){\line(0,1){10}}
\put(40,20){\line(0,1){10}}
\put(50,50){ $(b_{i},\lambda_i-i+b_i)$}
\put(53,45){\line(-1,-1){20}}
\put(50,20){\line(0,1){10}} \put(50,20){\line(1,0){10}}
\put(70,35){$(b_{i_{k}},\lambda_{i_{k}}-i_{k}+b_{i_{k}})$}
\put(66,35){\line(-1,-1){10}}
\put(-10,0){\line(0,1){70}} \put(-10,0){\line(1,0){70}}
\put(-5,10){$B(w)$}
\end{picture}
\caption{\label{caption:bfact2} Proof of Lemma~\ref{lemma:bfact2}}
\end{figure}

\medskip

Here, (a) is the case when
$b_i=b_{i_{k+1}}-i_{k+1}+i$, and (b) is the case when $b_i=b_{i_k}$.
It remains to show that ($b_i,\lambda_i-i+b_i$) has to be either on the vertical
boundary defined by $(b_{i_{k+1}},\lambda_{i_{k+1}}-i_{k+1}+b_{i_{k+1}})$, which is case (a),
or on horizontal boundary defined by $(b_{i_{k}},\lambda_{i_{k}}-i_{k}+b_{i_{k}})$, which is
case (b).

The only concern is if $(b_i,\lambda_i -i +b_i)$ appears strictly east of 
the vertical boundary in
case (a) or north of the horizontal boundary in case (b). However, this implies that $\Ess(w)$
contains a box not associated to a corner of $\lambda$, which contradicts Lemma~\ref{lemma:bfact1}.
\end{proof}

\section{Combinatorial formulae for multiplicity and Hilbert series of ${\mathcal O}_{e_v,X_w}$}

We now arrive at our formulae for multiplicity of $e_v\in X_w$ and the Hilbert series for ${\mathcal O}_{e_v,X_w}$
in the case $w$ is covexillary.

\begin{theorem}
\label{thm:flaggedrule}
${\rm mult}_{e_v}(X_w)$ counts the number of flagged semistandard Young tableaux of shape $\lambda(w)$ whose
rows are bounded by ${\bf b}(\Theta_{v,w})$.
\end{theorem}
\begin{proof}
This is immediate from Theorem~\ref{thm:standard}, Theorem~\ref{thm:prime} and Theorem~\ref{thm:basic}(V).
\end{proof}

The following result generalizes the determinantal formula (with the same proof) from
\cite{WYIII} for cograssmannian permutations:

\begin{theorem}
\label{thm:detrule}
We have the following expression for multiplicity as a determinant of a matrix with binomial coefficient
entries:
\[{\rm mult}_{e_v}(X_w)=\det\left({b_i+\lambda_i-i+j-1\choose \lambda_i-i+j}\right)_{1\leq i,j\leq \ell(\lambda)},\]
where $\ell(\lambda)$ is the number of nonzero parts of $\lambda$ and ${\bf b}={\bf b}(\Theta_{v,w})$.
\end{theorem}

The proof of Theorem~\ref{thm:detrule} is immediate from Theorem~\ref{thm:flaggedrule} once we have
discussed the determinantal expression for flagged Schur functions in Section~6.3.

\begin{example}
Continuing our example from the previous section, it follows from Example~\ref{exa:somepipes},
Example~\ref{exa:correspondtab} and Theorem~\ref{thm:flaggedrule} that the
multiplicity of $X_{7531462}$ at $e_{5123746}$ is $2$.
To illustrate Theorem~\ref{thm:detrule}, note that
since $\lambda_1=4, \lambda_2=2, \lambda_3=1$, $b_1=1, b_2=3, b_3=3$, Theorem \ref{thm:detrule} asserts that the multiplicity
$${\rm mult}_{e_v}(X_w)=
\begin{vmatrix}
{b_1+\lambda_1-1\choose \lambda_1}&{b_1+\lambda_1\choose \lambda_1+1}& {b_1+\lambda_1+1\choose \lambda_1+2}\\
{b_2+\lambda_2-2\choose \lambda_2-1}&{b_2+\lambda_2-1\choose \lambda_2}& {b_2+\lambda_2\choose \lambda_2+1}\\
{b_3+\lambda_3-3\choose \lambda_3-2}&{b_3+\lambda_3-2\choose \lambda_3-1}& {b_3+\lambda_3-1\choose \lambda_3}\end{vmatrix}
=\begin{vmatrix}
{4\choose 4}&{5\choose 5}& {6\choose 6}\\
{3\choose 1}&{4\choose 2}&{5\choose 3}\\
{1\choose -1}&{2\choose 0}&{3\choose 1}\end{vmatrix}
=\begin{vmatrix}1&1&1\\3&6&10\\0&1&3\end{vmatrix}=2,$$
in agreement with our previous computation.\qed
\end{example}

\begin{example}
A.~Woo~\cite{Woo:catalan} proved that when $w=(n+2)23\ldots (n+1)1$,
the multiplicity of the Schubert variety $X_{w}\subseteq {\rm
Flags}({\mathbb C}^{n+2})$ at the most singular point $e_{id}$
is given by the Catalan number $C_{n}=\frac{1}{n+1}{2n\choose n}$. Moreover, he conjectured that
the largest value multiplicity can attain for $v,w\in S_{n+2}$ is this Catalan number.
Woo's permutation is covexillary, and this multiplicity problem is
also solved by Theorem~\ref{thm:flaggedrule}. \qed
\end{example}

A richer invariant than multiplicity is the Hilbert series of ${\mathcal O}_{e_v,X_w}$.
In order to state our formula for it, recall the notion of flagged set-valued
semistandard tableaux from \cite{KMY}. A {\bf set-valued, semistandard filling} of $\lambda$ \cite{Buch:KLR}
is an assignment of non-empty sets to each box of $\lambda$ so that each entry of a box
is weakly smaller than each entry to its right, and strictly smaller than any entry
strictly below it. Such a filling is flagged by ${\bf b}={\bf b}(w)$ if each entry in
a row $i$ is at most $b_i$.

\begin{example}
Continuing Example~\ref{exa:correspondtab},
the additional flagged set-valued semistandard tableau for the flagging ${\bf b}={\bf b}(\Theta_{v,w})$
that is not (ordinary) semistandard is
\[\ktableau{{1}&{1}&{1}&{1}\\{2}&{2,3}\\{3}}.\]
\qed
\end{example}

Recall
\[T_{v,w}=\langle {\hat f}: {\hat f} \mbox{\ is the lowest (standard) degree component of $f\in I_{v,w}$}\rangle\]
is the (homogeneous) ideal of the projectivized tangent cone of ${\mathcal
N}_{v,w}$.

\begin{theorem}
\label{thm:Hilbert}
The ${\mathbb Z}$-graded Hilbert series of ${\mathcal O}_{e_v,X_w}$ and ${\mathbb C}[{\bf z}^{(v)}]/T_{v,w}$ are given
respectively by
\[{\rm Hilb}({\mathcal O}_{e_v,X_w},t):=\sum_{i\geq 0}\dim({\mathfrak m}_{e_v}^{i}/{\mathfrak m}_{e_v}^{i+1})t^i=G_{\lambda}(t)/(1-t)^{{n\choose 2}}\]
and
\[{\rm Hilb}({\mathbb C}[{\bf z}^{(v)}]/T_{v,w},t):=\sum_{i\geq 0}\dim\left({\mathbb C}[{\bf z}^{(v)}]/T_{v,w}\right)_i t^i=G_{\lambda}(t)/(1-t)^{\ell(w_0 v)},\]
where
\[G_{\lambda}(t)=\sum_{k\geq |\lambda|}(-1)^{k-|\lambda|}(1-t)^k\times \#{\rm SetSSYT}(\lambda,{\bf b},k)\]
and $\#{\rm SetSSYT}(\lambda,{\bf b},k)$ equals the number of flagged set-valued semistandard Young tableaux
of shape $\lambda$ with flag ${\bf b}={\bf b}(\Theta_{v,w})$ and which uses exactly $k$ entries.
\end{theorem}

\begin{remark}
As with Theorem~\ref{thm:basic}(VI) one can straightforwardly write down multigraded Hilbert series that takes into
account both the standard grading and the usual action grading. We leave this as a remark since 
one requires a bunch of prerequisites about double Grothendieck polynomials (for covexillary permutations) from
\cite{KMY, WYIII} that we do not need otherwise in the text. \qed
\end{remark}

\begin{remark}
A permutation $w$ is {\bf cograssmannian} if it has a unique ascent,
at position $d$, i.e., $w(k)<w(k+1)$ if and only if $k=d$. Each
cograssmannian $w$ is clearly also covexillary. Moreover, there is a
bijective correspondence between $\lambda\subseteq d\times (n-d)$
and these cograssmannian permutations. Under this correspondence,
the multiplicity of a Grassmannian Schubert variety
$X_{\lambda}\subseteq Gr(d,{\mathbb C}^n)$ at a torus fixed point
$e_{\mu}$ can be computed using Theorem~\ref{thm:flaggedrule}.
Geometrically, this follows from the fact that the natural ``forgetting subspaces''
projection $\pi:{\rm Flags}({\mathbb C}^n)\twoheadrightarrow Gr(k,{\mathbb C}^n)$
restricts to a locally trivial fibration $X_w\to X_{\lambda}$ with fiber $P/B$ where
$P$ is the maximal parabolic such that $G/P\cong Gr(d,{\mathbb C}^n)$; see \cite[Example~1.2.3]{Brion}. \qed
\end{remark}

\begin{remark}
It is natural to wonder about relations between multiplicity and the Kazhdan-Lusztig polynomial, as might be seen
by comparing our formulae with the Kazhdan-Lusztig polynomial formula for covexillary Schubert varieties \cite{Lascoux}. 
Small computations contraindicate any simple comparisons.
\end{remark}

\section{Proofs of the main theorems}

\subsection{Covexillary Schubert determinantal ideals}
    Let $\prec_{\rm antidiag}$ denote any term order that picks off the main antidiagonal (i.e., southwest to northeast main diagonal) term of
any minor of $Z$. We will use the following result:

\begin{theorem}[\cite{KMY}]
\label{thm:KMY}
Let $w\in S_n$ be covexillary.
The essential determinants of $I_{w}$ form a Gr\"{o}bner basis with respect to $\prec_{\rm antidiag}$.
Moreover, the initial ideal ${\rm init}_{\prec_{\rm antidiag}} I_{w}$ is reduced and equidimensional, with
prime decomposition
\[{\rm init}_{\prec_{\rm antidiag}}I_{w}=\bigcap_{\mathcal P} \langle z_{ij} : (i,j)\in {\mathcal P}\rangle,\]
where ${\mathcal P}$ is a pipe dream for $w$.

The Stanley-Reisner simplicial complex is a
homeomorphic to a vertex decomposable (and hence shellable) ball or sphere.

The interior faces of the complex are labeled by
set-valued semistandard Young tableaux of shape $\lambda$ and flagged by ${\bf b}(w)$, and the facets are labeled
by the subset of ordinary semistandard Young tableaux. The codimension $k$ interior faces are labeled by these
tableaux with $|\lambda|+k$ entries.
\end{theorem}

As is explained in \cite{KMY}, the irreducible components are in manifest bijection with pipe dreams for $w$:
the appearance of a generator $z_{ij}$ indicates the position of $+$'s, using the usual coordinates
consistent with our labeling of the generic matrix $Z$. Compare this with Theorem~\ref{thm:prime} where we
instead express things in terms of the variables of ${\widetilde Z}$.

Although our proof of Theorem~\ref{thm:standard} will use the Gr\"{o}bner basis theorem of \cite{KMY} for
Schubert determinantal ideals (recapitulated in Section~7.1), we remark that
Theorem~\ref{thm:standard} actually provides a generalization. This is based on the fact that
any Schubert determinantal ideal can be realized as a Kazhdan-Lusztig ideal, and this ideal is homogeneous with
respect to the standard grading; see \cite[Section~2.3]{WYIII}.

\subsection{Flagged Schur polynomials; proof of Theorem~\ref{thm:detrule}}
The weight generating series for semistandard tableaux with row entries {\bf flagged} (bounded) by a vector ${\bf b}$ is
called the {\bf flagged Schur polynomial}. An application of a standard Gessel-Viennot type argument
establishes that
\begin{equation}
\label{eqn:GesselViennot}
\det(h_{\lambda_i-i+j}(x_1,\ldots x_{b_i}))=\sum_{T\in \mathcal{T}(\lambda,\bf b)}{\bf x}^{{\rm wt}(T)},
\end{equation}
where $h_k(x_1,\ldots,x_m)$ is the complete homogeneous symmetric function on the variables
$x_1,\ldots,x_m$ and the right-hand side of the equality is by definition the flagged Schur polynomial,
where the sum runs over all semistandard tableau of shape $\lambda$ and flagged by ${\bf b}$.
See \cite[Cor~2.6.3]{Manivel}.

We are now ready to give our proofs of the determinantal expression for multiplicity:

\medskip
\noindent
\emph{Proof of Theorem~\ref{thm:detrule}:}
This is immediate from Theorem~\ref{thm:flaggedrule} combined with (\ref{eqn:GesselViennot}) evaluated
at $x_i=1$ for all $i$ and the fact $h_{\lambda_i -i +j}(1,1,\ldots,1)=
{b_i+\lambda_i-i+j-1\choose \lambda_i-i+j}$. \qed

Now suppose $X=(x_i)_{i\in I}$ and $Y=(y_j)_{j\in J}$ are two finite
families of indeterminates. Define polynomials $h_k(X-Y)$ by the power series expansion
\[\sum_{k\in {\mathbb Z}}u^kh_k(X-Y)=\prod_{j\in J}(1-uy_j)/\prod_{i\in I}(1-ux_i).\]
In the literature one finds the nomenclature {\bf flagged double Schur function}, which is defined by the
following determinantal expression  \cite[Definition
4.1]{ChenLiLouck},
\begin{equation}\label{eq:determinantal for flagged double Schur function}s_{\lambda,{\bf b}}(X-Y)=\det\big{(}h_{\lambda_i - i +j}(X_{b_i}-Y_{\lambda_i+b_i-i})\big{)}_{1\leq i,j\leq \ell},\end{equation}
where
\[X_{b_i}=(x_1,x_2,\dots,x_{b_i}), Y_{\lambda_i+b_i-i}=(y_1,\dots,y_{\lambda_i+b_i-i}).\]
There is also a tableau expression
\begin{equation}\label{eq:s_lambda,b}s_{\lambda,{\bf b}}(X-Y)=\sum_{T\in \mathcal{T}(\lambda,\bf b)}\prod_{\alpha\in T}\big{(}x_{T(\alpha)}-y_{T(\alpha)+C(\alpha)}\big{)}.
\end{equation}
where $C(\alpha)=c-r$ if $\alpha$ is in the $r$-th row and $c$-th
column.

Notice that by comparing the above formula with \cite[Theorem 5.8]{KMY},  the single (respectively, double) Schubert
polynomial ${\mathfrak S}_{w_0w}(X,Y)$ for a covexillary $w$ is the same as the single (respectively double) flagged Schur
polynomial of shape $\lambda(w)$ with flagging ${\bf b}(w)$.

\subsection{Hilbert series and an identity of flagged Schur polynomials}
We now use standard notions from combinatorial commutative algebra, found
in the textbook \cite{Miller.Sturmfels}.

Consider a polynomial ring $S=\mathbb{C}[z_1,\ldots,z_m]$
with a grading such
that $z_i$ has some degree $\mathbf{a}_i\in\mathbb{Z}^N$.  A finitely graded $S$-module
$M=\bigoplus_{{\bf v}\in {\mathbb Z}^N} M_{{\bf v}}$
has a free resolution
\[E_{\bullet}:0\leftarrow E_1\leftarrow E_2\leftarrow\cdots\leftarrow E_L\leftarrow 0\]
where
$E_i=\bigoplus_{i=1}^{\beta_j}S(-\mathbf{d}_{ij})$
is graded with the $j$-th summand of $E_i$ generated in degree
$\mathbf{d}_{ij}\in {\mathbb Z}^N$.

Then the (${\mathbb Z}^N$-graded) {\bf K-polynomial} of $M$ is
\[{\mathcal K}(M,{\bf t})=\sum_j(-1)^j\sum_i {\bf t}^{{\bf d}_{ij}}.\]
In any case where $S$ is {\bf positively graded},
meaning that the $\mathbf{a}_i$
generate a pointed cone in $\mathbb{Z}^N$, ${\mathcal K}(M,{\bf t})$ is the
numerator of the ${\mathbb Z}^N$-graded Hilbert series:
\[{\rm Hilb}(M,{\bf t})=\frac{{\mathcal K}(M,{\bf t})}
{\prod_{i}(1-\mathbf{t}^{\mathbf{a}_i})}.\]
The {\bf multidegree} ${\mathcal C}(M, {\bf t})$
is by definition the sum of the lowest degree terms of
${\mathcal K}(M,{\bf 1-t})$.  (This means we substitute $1-t_k$ for $t_k$ for
all $k$, $1<k<N$.) Also, if $X={\rm Spec}(S/I)$ then let
${\mathcal C}(X, {\bf t}):={\mathcal C}(S/I,{\bf t})$

\begin{proposition}
\label{prop:specialization}
Let $w$ be covexillary, $\lambda=\lambda(w)$ and ${\bf b}={\bf b}(w)$. Set
\[X=(t_{v(1)},\dots,t_{v(n)}), \  Y=(t_n,\dots,t_1)\]
and $s_{\lambda,{\bf b}}(X-Y)$ be the associated flagged Schur function. Then the following equality holds:
\begin{equation}
\label{eqn:tripleequalityG}
{\mathcal C}({\mathcal N}_{v,w},t_{ij}\mapsto t_{v(j)}-t_{n-i+1})=(-1)^{|\lambda|}s_{\lambda,{\bf b}}(Y-X).
\end{equation}
\end{proposition}
\begin{proof}
$$\aligned & \mathcal{C}(\mathcal{N}_{v,w},t_{ij}\mapsto
t_{v(j)}-t_{n-i+1})\\
&={\mathfrak S}_{w_0w}(X,Y) \quad \textrm{(by \cite[Theorem 4.5]{WYIII})}\\
&=\sum_{\mathcal P}\prod_{(i,j)\in {\mathcal P}}(x_j-y_i) \quad \textrm{(summing over pipe dreams ${\mathcal P}$ for $w$, by \cite{KMY})}\\
&=(-1)^{|\lambda|}\sum_{\mathcal P}\prod_{(i,j)\in {\mathcal P}}(y_i-x_j)  \quad \textrm{(summing over the same  ${\mathcal P}$'s as above)}\\
&=(-1)^{|\lambda|}\sum_{T\in \mathcal{T}(\lambda,\bf b)}\prod_{\alpha\in T}\big{(}y_{T(\alpha)}-x_{T(\alpha)+C(\alpha)}\big{)}
\quad \textrm{(under the correspondence of Section~5.2)}\\
&=(-1)^{|\lambda|}s_{\lambda,{\bf b}}(Y-X) \quad \textrm{(by the tableau formula (\ref{eq:s_lambda,b}))}.\endaligned$$
\end{proof}

\subsection{Conclusion of the proofs}
Let
\[I_{v,w}=\langle g_1,g_2,\ldots,g_N\rangle\]
where $\{g_1,\ldots,g_N\}$ are the essential determinants.
Also, let
\[J_{v,w}=\langle {\rm init}_{\prec_{v,w,\pi}} g_1, {\rm init}_{\prec_{v,w,\pi}} g_2, \ldots, {\rm init}_{\prec_{v,w,\pi}} g_N\rangle.\]
It is always true that
\[J_{v,w}\subseteq {\rm init}_{\prec_{v,w,\pi}}\ I_{v,w}.\]
Equality holds if and only if $\{g_1,\ldots,g_N\}$ is a Gr\"{o}bner basis with respect to $\prec_{v,w,\pi}$.
Let ${\widetilde Z}$ be the shuffled generic matrix determined by $(v\leq w)$,
as defined in Section~4.2.
Let
\[{\widetilde I}_{\Theta_{v,w}}\subseteq {\mathbb C}[{\widetilde Z}]\cong {\mathbb C}[{\bf z}]\]
be the Schubert determinantal ideal as defined by taking sub-determinants of the shuffled matrix ${\widetilde Z}$, as determined by the rank matrix for $\Theta_{v,w}$.
Let $\prec_{{\widetilde{\rm antidiag}}}$ denote a term order that picks off the (southwest to northeast) antidiagonal term
of any sub-determinant of ${\widetilde Z}$. Using Theorem~\ref{thm:KMY} we immediately
conclude that under $\prec_{\widetilde{{\rm antidiag}}}$,
the essential (or defining) minors of ${\widetilde Z}$ (coming from the rank conditions for $\Theta_{v,w}$)
are a Gr\"{o}bner basis for ${\widetilde I}_{\Theta_{v,w}}$ and hence the lead terms generate ${\rm init}_{\prec_{{\widetilde{\rm antidiag}}}} {\widetilde I}_{\Theta_{v,w}}$.

\begin{lemma}
\label{lemma:thesame} The aforementioned generators of ${\rm
init}_{\prec_{{\widetilde{\rm antidiag}}}} {\widetilde
I}_{\Theta_{v,w}}$ are a subset of the generators of $J_{v,w}$.
\end{lemma}
\begin{proof}
Consider an essential determinant $g$ of $I_{v,w}$, which is associated to an $r\times r$ minor $M$
of the submatrix $Z_{{\mathfrak e}}^{(v)}$ of $Z^{(v)}$
associated to ${\mathfrak e}\in \Ess(w)$; here $r=r_{\mathfrak e}^w$.
There are $r_{\mathfrak e}^v\leq r$ many $1$'s in $Z_{{\mathfrak e}}^{(v)}$,
by Proposition~\ref{prop:bruhat}. Assume that the minor $g$ uses all the rows and columns that these
$1$'s sit in. Note that since $g$ is homogeneous with respect to the usual action grading, so by Theorem~\ref{thm:basic}(I),
$\prec_{v,w,\pi}$ will choose the terms of lowest total degree
first, and so it will pick out all terms of the determinant that use all these $1$'s in their product.
Thus, by the definition of $\prec_{v,w,\pi}$ given in Section~4.2, the lead term will exactly be the antidiagonal term of the
minor of a submatrix of ${\widetilde Z}$, which, when the rows and columns are permuted, is precisely
the submatrix $M^{\circ}$ of $M$ that comes from striking out all the rows and columns of $M$
having $1$'s in them. Thus this lead term is a generator of ${\rm init}_{\prec_{{\widetilde{\rm antidiag}}}} {\widetilde I}_{\Theta_{v,w}}$. This generator corresponds to ${\mathfrak e}'\in \Ess(\Theta_{v,w})$, as defined in Definition-Lemma~\ref{deflemma:Theta}.
On the other hand, one can similarly see that all generators of ${\rm init}_{\prec_{{\widetilde{\rm antidiag}}}} {\widetilde I}_{\Theta_{v,w}}$
can be realized in this manner.
\end{proof}

Consider the natural projection
\[\varphi: \mathbb{C}[{\bf z}]\to \mathbb{C}[{\bf z}^{(v)}]\]
that sends all variables not in ${\bf z}^{(v)}$ to $0$. By Lemma~\ref{lemma:thesame}, all generators of
 ${\rm init}_{\prec_{{\widetilde{\rm antidiag}}}} {\widetilde I}_{\Theta_{v,w}}$ only use variables in ${\mathbb C}[{\bf z}^{(v)}]$
so it makes sense to define the ideal
\[H_{v,w}=\varphi\left(({\rm init}_{\prec_{\rm {\widetilde{antidiag}}}} {\widetilde I}_{\Theta_{v,w}})\mathbb{C}[{\bf z}^{(v)}]\right)\subseteq \mathbb{C}[{\bf z}^{(v)}].\]

%
%
%

\begin{lemma}
\label{lemma:equidandreduced}
${\rm Spec}({\mathbb C}[{\bf z}]/H_{v,w})$ defines an equidimensional and reduced scheme.
\end{lemma}
\begin{proof} As we have said,
Theorem~\ref{thm:KMY} implies ${\rm init}_{\prec_{{\widetilde{\rm antidiag}}}}{\widetilde I}_{\Theta_{v,w}}$
defines an equidimensional and reduced scheme. Since
dividing out an irrelevant factor of affine space does not affect these
properties, the claim holds.
\end{proof}

Often, when proving equality of two homogeneous ideals $A\subseteq B$ in a positively graded ring $R$,
one expects to show that the multigraded Hilbert series of $R/A$ and $R/B$ are equal. Fortunately, our
arguments will only require equality of multidegrees, thanks to the following:

\begin{lemma}[Lemma~1.7.5 of \cite{Knutson.Miller}]
\label{lemma:KM}
Let $I'\subseteq \Bbbk[z_1,\ldots,z_m]$ be an ideal homogeneous for a positive ${\mathbb Z}^d$-grading. Suppose
$H$ is an equidimensional radical ideal contained inside $I'$. If the zero schemes of $I'$ and $H$ have equal
multidegrees, then $I'=H$.
\end{lemma}

We will apply Lemma~\ref{lemma:KM} in the case $H=H_{v,w}\subseteq I'={\rm init}_{\prec_{v,w,\pi}}I_{v,w}$.

\begin{proposition}
\label{proposition:multidegreesagree} The multidegree of ${\mathcal
N}_{v,w}$ equals the multidegree of ${\mathbb C}[{\bf z}]/({\rm
init}_{\prec_{{\widetilde{\rm antidiag}}}} {\widetilde I}_{\Theta_{v,w}})$, each
with respect to the usual action of $T\subset GL_n$ (as defined in
Section~2.4).
\end{proposition}
\begin{proof} By Proposition \ref{prop:specialization}, the multidegree of
$\mathcal{N}_{v,w}$ is
\begin{equation}\label{eq:multideg N}\mathcal{C}(\mathcal{N}_{v,w},t_{ij}\mapsto
t_{v(j)}-t_{n-i+1})=(-1)^{|\lambda|}s_{\lambda,\bf{b}}(Y-X),\end{equation}
where ${\bf b}={\bf b}(w)$.

On the other hand, in \cite{KMY}, it was proved that, for
$w$ covexillary, the multidegree of the matrix Schubert variety
\[{\overline X}_{w}={\rm Spec}({\mathbb C}[{\bf z}]/I_w)\]
is the flagged Schur polynomial
\[s_{\lambda,{\bf b}}(Y-X) \mbox{\ where $\lambda=\lambda(w)$ and
${\bf b}={\bf b}(w)$.}\]
However, this multidegree is with respect to the $2n$-dimensional torus action where a vector
\[(a_1,\ldots,a_n,a'_1,\ldots,a'_n)\in T\times T\]
acts by rescaling row $i$ (from the top) of a matrix by $a_i^{-1}$
and rescaling column $i$ by $a'_i$. On the other hand there is an
embedding of tori
\begin{equation}
\label{eqn:tori1}
T\hookrightarrow T\times T:(a_1,\dots,a_n) \mapsto
(a_1,\dots,a_n; a_{v(1)},\dots,a_{v(n)})
\end{equation}
that realizes the usual torus action as a subtorus of $T\times T$.
As is explained in
\cite{WYIII}, because of this embedding, one can
compute the multidegree for ${\overline X}_w$ under the usual torus
action by the substitutions
$$\aligned & X=(x_1,\dots,x_n)=(t_{v(1)},\dots,t_{v(n)}),\\
&Y=(y_1,\dots,y_n)=(t_n,\dots,t_1).\endaligned$$
Let $\rho=r_1 \cdots r_n\in
S_n,\ \chi=c_1 \cdots c_n\in S_n$ be defined as in
\S\ref{subsection: The Mora basis theorem}. Set
$$\aligned
&X'=(x_1',\dots,x_n')=(x_{c_1},\dots,x_{c_n})=(t_{v(c_1)},\dots,t_{v(c_n)}),\\
&Y'=(y_1',\dots,y_n')=(y_{r_1},\dots,y_{r_n})=(t_{n+1-r_1},\dots,t_{n+1-r_n}),\\
&\bf{b}''={\bf b}(\Theta_{v,w}).\\
\endaligned$$
There is another embedding of tori
\begin{equation}
\label{eqn:tori2}
T\times T\hookrightarrow T\times T:(a_1,\dots,a_n,a_1',\dots,a_n')\mapsto (a_{r_1},\dots,a_{r_n},a_{c_1}',\dots,a_{c_n}')
\end{equation}
Composing the two tori embeddings (\ref{eqn:tori1}) and (\ref{eqn:tori2}) allows us to twist the usual action grading to one
on ${\rm Fun}[{\widetilde Z}]\cong {\mathbb C}[{\bf z}]$. Putting this together, the multidegree of the matrix Schubert variety
$\mathbb{C}[{\bf z}]/({\rm init}_{\prec_{{\widetilde{\rm
antidiag}}}} {\widetilde I}_{\Theta_{v,w}})$ is
\begin{equation}\label{eq:multideg initI}{\mathcal C} \left(\mathbb{C}[{\bf
z}]/({\rm init}_{\prec_{{\widetilde{\rm antidiag}}}} {\widetilde
I}_{\Theta_{v,w}}),{\bf t}\right)
=(-1)^{|\lambda|}s_{\lambda,\bf{b}''}(Y'-X'),
\end{equation}
with respect to the grading ${\rm deg }
(z_{ij})=t_{v(j)}-t_{n+1-i}$. Here, $\lambda=\lambda(w)=\lambda(\Theta_{v,w})$, see Definition-Lemma~\ref{deflemma:Theta}.

 In order to prove that the two multidegrees (\ref{eq:multideg
N}) and (\ref{eq:multideg initI}) are equal polynomials, we define
an auxiliary flagging
\[{\bf{b}}'=(b'_1,\dots,b'_\ell), \mbox{\ where for
$1\le i\le \ell$, $b'_i=b_i-k_i$}\]
and
\[k_i=\#\{\mbox{$1$'s in $Z^{(v)}$ that are dominated by $(b_i, \lambda_i-i+b_i)$}\},\]
cf.~Lemma~\ref{lemma:simple}. We will instead establish
\begin{equation}\label{eq:b=b'}s_{\lambda,\bf{b}}(Y-X)=s_{\lambda,\bf{b}'}(Y'-X')\end{equation}
 and
\begin{equation}\label{eq:b'=b''}s_{\lambda,\bf{b}'}(Y'-X')=s_{\lambda,\bf{b}''}(Y'-X'),\end{equation}
from which the equality follows.

We now prove (\ref{eq:b=b'}). By (\ref{eq:determinantal for
flagged double Schur function}), it is equivalent to prove
$$\det\big{(}h_{\lambda_i - i +j}(Y_{b_i}-X_{\lambda_i+b_i-i})\big{)}_{1\leq i,j\leq
\ell}= \det\big{(}h_{\lambda_i - i
+j}(Y'_{b'_i}-X'_{\lambda_i+b'_i-i})\big{)}_{1\leq i,j\leq \ell},$$
i.e.,
$$\det\Bigg{(}[u^{\lambda_i-i+j}]\frac{\prod_{x\in
X_{\lambda_i-i+b_i}}(1-xu)}{\prod_{y\in
Y_{b_i}}(1-yu)}\Bigg{)}_{1\le i,j\le \ell}
=\det\Bigg{(}[u^{\lambda_i-i+j}]\frac{\prod_{x\in
X'_{\lambda_i-i+b'_i}}(1-xu)}{\prod_{y\in
Y'_{b'_i}}(1-yu)}\Bigg{)}_{1\le i,j\le \ell}.$$
In fact, more strongly we show that for every $1\le i\le \ell$,
\begin{equation}\label{eq:multideg}\frac{\prod_{x\in X_{\lambda_i-i+b_i}}(1-xu)}{\prod_{y\in
Y_{b_i}}(1-yu)}= \frac{\prod_{x\in
X'_{\lambda_i-i+b'_i}}(1-xu)}{\prod_{y\in Y'_{b'_i}}(1-yu)}.
\end{equation}
The equality (\ref{eq:multideg}) is proved as follows. We use the
notation as in Section \ref{subsection: The Mora basis theorem}.
Recall (\ref{eqn:deletedspots}); we now define
$$A_i=\left\{t_{v(\beta^{(i)}_j)}\right\}_{1\le j\le
k_i}\subseteq\{t_1,\dots,t_n\}.$$ By Lemma \ref{lemma:simple}, we
have the following equalities of subsets of $\{t_1,\dots,t_n\}$:
\begin{equation}\label{eq:Y}
\aligned
Y_{b_i}&=\{y_1,\dots,y_{b_i}\}\\
&=\{t_n,\dots,t_{n+1-b_i}\}\\
&=\{t_{n+1-r_1},\dots,t_{n+1-r_{b'_i}}\}\cup\{t_{n+1-\alpha^{(i)}_1},\dots,t_{n+1-\alpha^{(i)}_{k_i}}\}\\
&=\{t_{n+1-r_1},\dots,t_{n+1-r_{b'_i}}\}\cup\{t_{v(\beta^{(i)}_1)},\dots,t_{v(\beta^{(i)}_{k_i})}\}\\
&=\{y'_1,\dots,y'_{b'_i}\}\cup A_i\\
&=Y'_{b'_i}\cup A_i.\endaligned
\end{equation}
\begin{equation}\label{eq:X}
\aligned
X_{\lambda_i-i+b_i} &=\{x_1,\dots,x_{\lambda_i-i+b_i}\}\\
 &=\{t_{v(1)},\dots,t_{v(\lambda_i-i+b_i)}\}\\
 &=\{t_{v(c_1)},\dots,t_{v(c_{\lambda_i-i+b'_i})}\}\cup\{t_{v(\beta^{(i)}_1)},\dots,t_{v(\beta^{(i)}_{k_i})}\}\\
 &=\{x'_1,\dots,x'_{\lambda_i-i+b'_i}\}\cup A_i\\
 &=X'_{\lambda_i-i+b'_i}\cup A_i,\endaligned
\end{equation}

Because of (\ref{eq:Y}) and (\ref{eq:X}), we can cancel out the
factors
\[\left(1-t_{v(\beta^{(i)}_j)}u\right), \mbox{\ for $1\le j\le k_i$,}\]
which appear both in the numerator and in the denominator on the
left-hand side of (\ref{eq:multideg}). After the cancellation, we
obtain the the right-hand side of (\ref{eq:multideg}). This proves
(\ref{eq:multideg}).

Next, we prove (\ref{eq:b'=b''}) using the tableau
formula for flagged double Schur functions.
By the discussion of
Section~5.2, it suffices to show that the two sets of flagged semistandard Young tableaux
of shape $\lambda$, flagged by ${\bf b'}$
and ${\bf b''}$ respectively, are the same.

Let
\[\{i_1,i_2,\dots,i_m\}=\{i\;|\; 1\le i\le \ell,\; \lambda_i>\lambda_{i+1}\}\]
(assume $\lambda_{\ell+1}=0$ and  $i_1<i_2<\cdots<i_m$). In other
words, this is the set of indices of the rows of the Young diagram of $\lambda$ that
have corners on their right ends.

We claim that $b''_i\le b'_i$, with
equality when
$i=i_1,\dots,i_m$. By Lemma~\ref{lemma:bfact1} we have
\[\Ess(w)=\{(b_i,\lambda_i-i+b_i)\;|\;i=i_1,\dots,i_m\},\]
and
\[\Ess(\Theta_{v,w})=\{(b''_i,\lambda_i-i+b''_i)\;|\;i=i_1,\dots,i_m\}.\]
Therefore for
$i=i_1,\dots,i_m$, we have, from the definition of
${\bf b'}={\bf b}-(k_1,k_2,\ldots, k_\ell)$ and ${\bf b''}={\bf b}(\Theta_{v,w})$ that
\begin{equation}
\label{eqn:vuseful}
b'_i=b_i-(\mbox{the number of $1$'s in $Z^{(v)}$ dominated by $(b_i,\lambda_i-i+b_i)$})=b''_i.
\end{equation}
On the other hand, for
\[i\in\{1,\dots,\ell\}\setminus\{i_1,\dots,i_m\},\]
let $k$ be the
index that
\[0\le k\le m-1 \mbox{\ and  \ } i_k<i<i_{k+1}\]
(declare $i_0=0$).

By Lemma~\ref{lemma:bfact2} we have
\[b_i=\max(b_{i_{k+1}}-i_{k+1}+i,b_{i_k})\]
(define $b_{i_0}=b_0=0$), and
\[b''_i=\max(b''_{i_{k+1}}-i_{k+1}+i,b''_{i_k}),\]
(where $b''_{i_0}=b''_0=0$).

Therefore the claimed inequality
\[b''_i\le b'_i,\]
or
\[b_i-b''_i\ge b_i-b'_i,\]
is equivalent to
\begin{multline}\label{eq:inequality b'b''}
\max(b_{i_{k+1}}-i_{k+1}+i,b_{i_k})-\max(b''_{i_{k+1}}-i_{k+1}+i,b''_{i_k})\ge\\
\mbox{\{number of 1's in $Z^{(v)}$ dominated by $(b_i,\lambda_i-i+b_i)$\}}.
\end{multline}
This can be checked in two cases.

\noindent Case (1): $b''_{i_{k+1}}-i_{k+1}+i\le b''_{i_k}$.

In this
case, the left-hand side of (\ref{eq:inequality b'b''}) is
$$\aligned &b_i-b''_{i_k}=(b_i-b_{i_k})+(b_{i_k}-b''_{i_k})\\
&=(b_i-b_{i_k})+\#\textrm{\{$1$'s in $Z^{(v)}$ dominated by $(b_{i_k},\lambda_{i_k}-i_k+b_{i_k})\}$} \\
&\ge\#\textrm{\{$1$'s in $Z^{(v)}$ dominated by
$(b_i,\lambda_i-i+b_i)$\}}\endaligned$$
where the second equality is by (\ref{eqn:vuseful}) and
the last inequality is
because, in the rows $b_{i_k}+1,\dots,b_i$, there are at
most $(b_i-b_{i_k})$ many $1$'s in $Z^{(v)}$ dominated by
$(b_i,\lambda_i-i+b_i)$, and $\lambda_i -i +b_i\leq \lambda_{i_k}-i_k+b_{i_k}$
since $i_k<i$; see (\ref{eqn:someineq}).

\noindent Case (2): $b''_{i_{k+1}}-i_{k+1}+i> b''_{i_k}$.

In this
case, the left-hand side of (\ref{eq:inequality b'b''}) is
$$\aligned &b_i-(b''_{i_{k+1}}-i_{k+1}+i)=(b_{i_{k+1}}-b''_{i_{k+1}})+(b_i-b_{i_{k+1}}+i_{k+1}-i)\\
&=\#\textrm{\{$1$'s in $Z^{(v)}$ dominated by $(b_{i_{k+1}},\lambda_{i_{k+1}}-i_{k+1}+b_{i_{k+1}})\}$}+(b_i-b_{i_{k+1}}+i_{k+1}-i) \\
&\ge\#\textrm{\{$1$'s in $Z^{(v)}$ dominated by
$(b_i,\lambda_i-i+b_i)$\}}\endaligned$$
where the second equality is by (\ref{eqn:vuseful}) and
the last inequality is
because, in the columns
$(\lambda_{i_{k+1}}-i_{k+1}+b_{i_{k+1}}+1),
(\lambda_{i_{k+1}}-i_{k+1}+b_{i_{k+1}}+2),\dots,(\lambda_i-i+b_i)$,
there are at most
$$\aligned
(\lambda_i-i+b_i)-(\lambda_{i_{k+1}}-i_{k+1}+b_{i_{k+1}}) & =  (b_i-b_{i_{k+1}}+i_{k+1}-i)+(\lambda_i - \lambda_{i_{k+1}})\\
& =  b_i-b_{i_{k+1}}+i_{k+1}-i\endaligned$$
many $1$'s in $Z^{(v)}$ dominated by $(b_i,\lambda_i-i+b_i)$. (We have again applied (\ref{eqn:someineq}).)

Lastly, we show that the flagging ${\bf b}'$ and ${\bf b}''$ give
the same set of Young tableaux. In other words,
we need to show that any semistandard Young tableau
of shape $\lambda$ flagged by ${\bf b}'$ is also flagged by ${\bf b}''$.
Since $b_{i_k}'=b_{i_k}''$ for all $k$, it remains to consider
$i$ that satisfies $i_k<i<i_{k+1}$ for some $0\le k\le m-1$ (again define
$i_0=0$). Since $\lambda_i=\lambda_{i+1}=\cdots=\lambda_{i_{k+1}}$,
the length of rows $i$, $i+1$,$\dots$,$i_{k+1}$ of the Young tableau
are the same. Denote by ${\tt label}(i,j)$ the entry at the $i$-th
row and $j$-th column of the Young tableau. Then by the definition
of semi-standard Young tableaux,
$${\tt label}(i,\lambda_i)< {\tt label}(i+1,\lambda_{i+1})<{\tt label}(i+2,\lambda_{i+2})<\cdots<{\tt label}(i_{k+1},\lambda_{i_{k+1}})\le b''_{i_{k+1}},$$
hence $${\tt label}(i,\lambda_i)\le b''_{i_{k+1}}-(i_{k+1}-i)\le
\max(b''_{i_{k+1}}-(i_{k+1}-i),b''_{i_k})=b''_i.$$
So the condition ${\tt
label}(i,\lambda_i)\le b'_i$ is equivalent to the condition ${\tt
label}(i,\lambda_i)\le \min(b'_i,b''_i)=b''_i$. Therefore $\bf{b}'$
and $\bf{b}''$ give the same set of Young tableaux, as desired.
\end{proof}

\begin{proof}[Proof of Theorem \ref{thm:standard}, Theorem~\ref{thm:prime} and Theorem~\ref{thm:Hilbert}]
We know that
\[H_{v,w}\subseteq J_{v,w}\subseteq {\rm init}_{\prec_{v,w,\pi}} I_{v,w}\]
so  $H_{v,w}$ is an equidimensional, radical ideal contained inside ${\rm init}_{\prec_{v,w,\pi}} I_{v,w}$
by Lemmas \ref{lemma:thesame} and~\ref{lemma:equidandreduced}.
 The Hilbert series and hence multidegrees of
${\mathbb C}[{\bf z}^{(v)}]/{\rm init}_{\prec_{v,w,\pi}} I_{v,w}$ and ${\mathbb C}[{\bf z}^{(v)}]/I_{v,w}$
are equal by Theorem~\ref{thm:basic}(IV). On the other hand, these multidegrees are equal to the
multidegree of ${\mathbb C}[{\bf z}]/H_{v,w}$ by
Proposition~\ref{proposition:multidegreesagree} and the fact that multidegrees are unaffected by crossing the
scheme  by affine space. Hence
\[H_{v,w}=J_{v,w}={\rm init}_{\prec_{v,w,\pi}} I_{v,w}\]
by
Lemma~\ref{lemma:KM}. This proves the Theorem~\ref{thm:standard}.

Moreover, since $H_{v,w}={\rm init}_{\prec_{v,w,\pi}} I_{v,w}$ and since
${\rm Spec}\left({\mathbb C}[{\bf z}^{(v)}]/H_{v,w}\right)$
and  ${\rm init}_{\prec_{\widetilde{\rm antidiag}}} {\overline X}_{\Theta_{v,w}}$ only differ by crossing by affine space it follows the
prime decomposition of ${\rm init}_{\prec_{v,w,\pi}} I_{v,w}$ lifts to a prime decomposition 
of ${\rm init}_{\prec_{\widetilde{\rm antidiag}}} {\widetilde I}_{\Theta_{v,w}}$ and so Theorem~\ref{thm:prime}
follows from the prime decomposition theorem of \cite{KMY} (taking into account the permutation of coordinates).

Finally, Theorem~\ref{thm:Hilbert} follows from Theorem~\ref{thm:basic}(VI), Theorem~\ref{thm:KMY}, and the
discussion above.
\end{proof}

\section{Conjectures and final remarks}
We now present some conjectures that complement Theorem~\ref{thm:basic}. 

\begin{conjecture}
\label{conj:firstconj}
For some $\pi$, ${\rm init}_{\prec_{v,w,\pi}}I_{v,w}$ defines a reduced and
equidimensional scheme, i.e., the hypothesis and hence conclusion of Theorem~\ref{thm:basic} (V) holds.
\end{conjecture}

Specifically, consider the {\bf SE-NW shuffling} $\pi_{\nwarrow}$ that orders the variables by
reading columns right to left and bottom to top. Based on the results of \cite{Knutson.Miller} and \cite{WYIII} as well
as some computation (exhaustively for $n\leq 6$, as well as many random examples for $n\leq 10$), we believe that
this choice always satisfies Conjecture~\ref{conj:firstconj}.

However, we actually desire a choice of $\pi$ that, in some sense,
gives the neatest combinatorics; the reducedness and equidimensionality offered by Conjecture~\ref{conj:firstconj} merely provides a necessary criterion. 
Let us call $\pi$ {\bf generalized antidiagonal} if, after some
permutation of the rows, and separately, some permutation of the columns, of $Z$, then $\pi$ induces a pure lexicographic ordering
on ${\mathbb C}[{\bf z}]$ that favors the antidiagonal (southwest-northeast) term of any sub-determinant of the shuffled generic matrix ${\widetilde Z}$.
We can extend this definition to shufflings $\pi$ for the variables of $Z^{(v)}$ in the obvious way. Now call $\prec_{v,w,\pi}$ generalized antidiagonal
if $\pi$ is for $Z^{(v)}$. The {\bf SW-NE shuffling} $\pi_{\nearrow}$ that orders variables by
reading rows bottom to top and left to right induces such a term order. Also, the same is true for $\pi_{\nwarrow}$. 
However, our main motivation for these definitions comes from the fact that the 
term order of Section~4.2 is also generalized antidiagonal. On the other hand, we have:

\begin{example}
\label{exa:counter}
In general, not all choices of generalized antidiagonal $\pi$ satisfy Conjecture~\ref{conj:firstconj}.
In particular, if we consider the Schubert determinantal ideal for $w=563412$ using $\pi_{\nearrow}$,
the limit scheme is not reduced. (As we have said (cf. Section 2.3), Schubert determinantal ideals are special cases of
Kazhdan-Lusztig ideals.) However, it is reduced if one uses $\pi_{\nwarrow}$ (either by direct computation, or by
the Gr\"{o}bner basis theorem of \cite{Knutson.Miller}). \qed
\end{example}

Notwithstanding Example~\ref{exa:counter}, we expect that among the generalized antidiagonal $\pi$'s, there exists a choice that
not only satisfies Conjecture~\ref{conj:firstconj}, but whose $\pi$-shuffled tableaux exhibit ``good'' combinatorial
features. This assertion is consistent with our covexillary work, as well as the results of \cite{Knutson.Miller, KMY, WYIII}.

\begin{problem}
Find a Gr\"{o}bner basis with squarefree initial terms for $I_{v,w}$, with respect to (any of) the orders $\prec_{v,w,\pi}$ satisfying Conjecture~\ref{conj:firstconj}.
\end{problem}

One cannot always use the defining (or essential) minors of $I_{v,w}$:

\begin{example}
Consider $w=45231$ and $v=23451$. Then the defining minors give
$$I_{v,w}=
\left\langle
z_{11},\
\left|\begin{matrix}
z_{31} & 1 & 0\\
z_{21} & z_{22} & 1\\
z_{11} & z_{12} & z_{13}
\end{matrix}\right|
\right\rangle=
\langle z_{11},z_{11}-z_{13}z_{21}-z_{12}z_{31}+z_{13}z_{22}z_{31}\rangle.
$$
If they formed a Gr\"{o}bner basis with respect to some
$\prec_{v, w,\pi}$, we would have ${\rm init}_{\prec_{v,w,\pi}}
I_{v,w}=\langle z_{11}\rangle$. However,
$-z_{13}z_{21}-z_{12}z_{31}+z_{13}z_{22}z_{31}$ is in $I_{v,w}$ and
its initial term is $-z_{13}z_{21}$, $-z_{12}z_{31}$ or
$z_{13}z_{22}z_{31}$, none of which are in the ideal $\langle
z_{11}\rangle$. So ${\rm init}_{\prec_{v,w,\pi}} I_{v,w}\supsetneq
\langle z_{11}\rangle$, therefore the two stated generators do not
form a Gr\"{o}bner basis.\qed
\end{example}

If $V$ is a vertex of a simplicial complex $\Delta$ one can speak of the {\bf deletion}
and the {\bf link}:
\[{\rm del}_V(\Delta)=\{F\in \Delta: V\not\in F\}, \ \
{\rm link}_V(\Delta)=\{F\in {\rm del}_V(\Delta): \{V\}\cup F\in \Delta\},\]
as well as the {\bf star} of $V$, which is the cone of the link from $V$:
${\rm star}_V(\Delta)=\{F\in \Delta: \{V\}\cup F\in\Delta\}$.
One has the {\bf vertex decomposition} $\Delta={\rm star}_V(\Delta)\cup {\rm del}_{V}(\Delta)$.
L.~Billera and S.~Provan \cite{Billera.Provan} defined what it means for $\Delta$ to
be {\bf vertex decomposable}. By definition, every simplex is vertex decomposable, and in general,
$\Delta$ is vertex decomposable if and only if it is pure and have a vertex decomposition where the deletion
and link are vertex decomposable. They proved that a vertex decomposable complex is shellable (and therefore
Cohen-Macaulay). The conjecture below gives one possible feature of a ``good'' choice of $\pi$:

\begin{conjecture}
\label{conj:secondconj}
There exists a choice of $\pi$ among those satisfying Conjecture~\ref{conj:firstconj}
such that the Stanley-Reisner simplicial complex $\Delta_{v,w}$ associated to ${\rm init}_{\prec_{v,w,\pi}} {\mathcal N}_{v,w}$
is homeomorphic to a vertex decomposable, and hence shellable, ball or sphere (in particular the
hypothesis of Theorem~\ref{thm:basic}(VI) holds).
\end{conjecture}

Our faith in Conjecture~\ref{conj:secondconj} is mainly based on the results of \cite{Knutson.Miller, KMY, WYIII}, and our
covexillary results. We also have some limited experimental evidence for the conjecture.
We computationally checked implications of the conjecture (Cohen-Macaulayness and connectedness of
$\Delta_{v,w,\pi}$, whether $\Delta_{v,w,\pi}$ has the homology of a ball/sphere, nonnegative $h$-vector,
and that each codimension one face is contained in at most two facets), using the
shufflings $\pi_{\nearrow}$ and $\pi_{\nwarrow}$. We computed these exhaustively for $n\leq 5$ and for the majority of 
$n\leq 6$ (where already the computational demands are high), as well as some larger cases.


\begin{problem}
For which $\pi$ does the conclusion of Conjecture~\ref{conj:secondconj} hold?
\end{problem}

\begin{example}
Even for choices of $\pi$ such that Conjecture~\ref{conj:firstconj} holds, Conjecture~\ref{conj:secondconj} is not
always satisfied. Looking at the implication ``each codimension one face is contained in at most two facets'',
if we utilize $\pi_{\nearrow}$, this holds for $n=6$ in all cases except when $w=563412$, and
$v=123456, 123546, 132456$ or $132546$, whence $\Delta_{v,w,\pi_{\nearrow}}$ cannot always be a ball.
On the other hand, if one chooses $\pi_{\nwarrow}$, the implication is satisfied on these examples, but
not on others, say $w=563412$, $v=123546$.\qed
\end{example}

Our covexillary results also motivate the next two problems, which indicate
successive refinements of Conjecture~\ref{conj:secondconj}:

\begin{problem}
\label{problem:subword}
When does there exist a choice of $\pi$ such that the Stanley-Reisner complex $\Delta_{v,w}$ associated to ${\rm init}_{\prec_{v,w,\pi}} {\mathcal N}_{v,w}$ is homeomorphic to a subword complex 
$\Delta(Q,\sigma)$, as introduced by
\cite{Knutson.Miller:subword}?
\end{problem}

We refer the reader to \cite{Knutson.Miller:subword} for the definition of subword complexes, and where it was
established that they
are vertex decomposable and homeomorphic to balls/spheres. The facets are indexed by subwords of the fixed
word $Q=(i_1,i_2,\ldots,i_M)$ of length $\ell(\sigma)$ such that $s_{i_{j_1}}\cdots s_{i_{j_{\ell}}}=\sigma$
where $s_k=(k\leftrightarrow k+1)$ is a simple reflection in a symmetric group.
Thus, one hopes for a combinatorial recipe $(v,w)\mapsto (Q,\sigma)$ that solves the multiplicity problem.

\begin{problem}
\label{problem:specific}
When does there exist a choice of $\pi$ such that the ${\rm init}_{\prec_{v,w,\pi}} {\mathcal N}_{v,w}$ is
equal, after crossing by an appropriate affine space and permutation of coordinates,
to the limit of a matrix Schubert variety, or another Kazhdan-Lusztig variety,
under the Gr\"{o}bner degeneration of \cite{Knutson.Miller}, \cite{KMY}, \cite{Knutson:patches} 
and/or \cite{WYIII}?
\end{problem}

Since the Stanley-Reisner complexes of the stated Gr\"{o}bner limits are subword
complexes,
Problem~\ref{problem:subword} is solved by Problem~\ref{problem:specific}.

Finally, our covexillary results suggest an affirmative answer to:

\begin{problem}
Is multiplicity of $e_v\in X_w$ (and/or the Hilbert series of ${\mathcal O}_{e_v,X_w}$) independent of characteristic?
\end{problem}

\section*{Acknowledgements}
We thank Sara Billey, Alex Fink, Allen Knutson, Kevin Purbhoo, Bruce Sagan, Hal Schenck, Alexandra Seceleanu, David Speyer,
Michelle Wachs and Alexander Woo for helpful conversations.
We especially thank Ezra Miller for answering a number of literature questions.
AY thanks Alexander Woo for lessons learned during a related sequence of projects.
AY is partially supported by NSF grants DMS-0601010 and DMS-0901331. We made extensive
use of the computer algebra system {\tt Macaulay~2} during our investigations.

\end{document}